\documentclass{amsart}
\usepackage{amsmath, amssymb, graphics, setspace}
\usepackage[dvips]{graphicx}
\title[Duality on generalized cuspidal edges]{%
 Duality on generalized cuspidal edges preserving 
 singular set images and 
 first fundamental forms 
}
\date{July 29, 2020}
\author{Atsufumi Honda}
\address[Atsufumi Honda]{
Department of Applied Mathematics, 
Faculty of Engineering, Yokohama National University,
79-5 Tokiwadai, Hodogaya, Yokohama 240-8501, Japan
}
\email{honda-atsufumi-kp@ynu.ac.jp}

\author{Kosuke Naokawa}
\address[Kosuke Naokawa]{%
Department of Computer Science, 
Faculty of Applied Information Science,
Hiroshima Institute of Technology,  
2-1-1 Miyake, Saeki, Hiroshima, 731-5193, Japan
}
\email{k.naokawa.ec@cc.it-hiroshima.ac.jp}

\author{Kentaro Saji}
\address[Kentaro Saji]{%
  Department of Mathematics,
  Faculty of Science,
  Kobe University,
  Rokko, Kobe 657-8501}
\email{saji@math.kobe-u.ac.jp}
\author{Masaaki Umehara}
\address[Masaaki Umehara]{%
  Department of Mathematical and Computing Sciences,
  Tokyo Institute of Technology,
  Tokyo 152-8552, Japan}
\email{umehara@is.titech.ac.jp}
\author{Kotaro Yamada}
\address[Kotaro Yamada]{%
  Department of Mathematics,
  Tokyo Institute of Technology,
  Tokyo 152-8551, Japan}
\email{kotaro@math.titech.ac.jp}

\dedicatory{Dedicated to Professor 
Toshizumi Fukui for his sixtieth birthday.}


\thanks{%
  The first author was partially supported by the
  Grant-in-Aid for Young Scientists (B), No.~16K17605,
  the second author 
  was partially supported by 
  Grant-in-Aid for Young Scientists (B), No.~17K14197,
  and the third author
  was 
  partially supported by 
  (C) No.\ 18K03301 from JSPS.
  The forth author was partially 
  supported by the Grant-in-Aid for 
  Scientific Research (A) No.\ 26247005, 
  and the fifth author 
  was partially 
  supported by the Grant-in-Aid for 
  Scientific Research (B) No.\ 17H02839.
}%
\newcommand{\op}[1]{{\operatorname{#1}}}
\newcommand{\vect}[1]{{\boldsymbol{#1}}}
\newcommand{\ccr}{\op{ccr}}
\newcommand{\tccr}{\op{\tiny ccr}}
\newcommand{\R}{\boldsymbol{R}}

\newcommand{\mc}[1]{{\mathcal #1}}
\newcommand{\mb}[1]{{\mathbf #1}}

\newcommand{\pmt}[1]{{\begin{pmatrix} #1  \end{pmatrix}}}

\renewcommand{\phi}{\varphi}
\renewcommand{\epsilon}{\varepsilon}
\newcommand{\dy}{\displaystyle}

\renewcommand{\det}{\op{det}}
\ifx\EMS\undefined

\else
\fi

\numberwithin{equation}{section}
\newtheorem{theorem}{Theorem}[section]
\newtheorem{proposition}[theorem]{Proposition}
\newtheorem{corollary}[theorem]{Corollary}
\newtheorem{lemma}[theorem]{Lemma}
\newtheorem{fact}[theorem]{Fact}
\theoremstyle{definition}
\newtheorem{defi}[theorem]{Definition}
\newtheorem{remark}[theorem]{Remark}

\newtheorem{example}[theorem]{Example}
\newtheorem*{ack}{Acknowledgements}

\begin{document}
\maketitle
\begin{abstract}
In the second, fourth and fifth authors'
previous work, a duality on generic real analytic 
cuspidal edges in the Euclidean 3-space  $\R^3$
preserving their singular set images and
first fundamental forms, was given.
Here, we call this an \lq\lq isometric duality\rq\rq.
When the singular set image has no symmetries
and does not lie in a plane, the dual 
cuspidal edge is not congruent to 
the original one.
In this paper, we show that this duality
extends to generalized cuspidal edges
in $\R^3$, including cuspidal cross caps,
and 5/2-cuspidal edges.
Moreover, we give several new geometric
insights on this duality.
\end{abstract}

\section*{Introduction}
Consider a generic cuspidal edge germ $f$ whose singular set 
image is a given space curve $C$.
In the second, fourth and fifth authors'
previous work \cite{NUY}, 
the existence of an isometric dual $\check f$ of $f$
was shown, which is a cuspidal edge germ
having the same first fundamental form 
as $f$. Roughly speaking, a cuspidal edge which 
has the same first fundamental form 
and the same singular set image as $f$
but is
not right equivalent to $f$, is called an \lq\lq isomer'' of $f$
(see Definition \ref{def:IS} for details).
The isometric dual $\check f$ is a typical example
of isomers of $f$. Recently, the authors found that 
if we reverse the orientation of $C$,
two other candidates of isomers of $f$ denoted by
$f_*$ and $\check f_*$ are obtained
by imitating the construction of $\check f$.
These two map germs $f_*$ and $\check f_*$ are 
cuspidal edge germs which are called
the {\it inverse} and  the {\it inverse dual} of $f$, respectively 
($\check f_*$ is just the isometric dual of $f_*$).
In this paper, we will show that all of isomers of $f$
are right equivalent to one of 
$$
\check f,\,\, f_*,\,\, \check f_*.
$$
We will also determine the number of congruence classes 
in the set of 
isomers of $f$.

\medskip
By the terminology \lq\lq$C^r$-differentiable\rq\rq\
we mean $C^\infty$-differentiability if $r=\infty$
and real analyticity if $r=\omega$.
We denote by $\R^3$ the Euclidean $3$-space.
Let $U$ be a neighborhood of the origin $(0,0)$
in the $uv$-plane $\R^2$, and let
$f:U\to \R^3$ be a $C^r$-map.
Without loss of generality, we may assume $f(o)=\mb 0$,
where
\begin{equation}\label{eq:zero}
o:=(0,0),\qquad \mb 0:=(0,0,0).
\end{equation}
A point $p\in U$ is called a {\it singular point}
if $f$ is not an immersion at $p$.
A singular point $p\in U$ is called a {\it cuspidal edge point}
(resp. a {\it generalized cuspidal edge point}) 
if there exist local $C^r$-diffeomorphisms
$\phi$ on $\R^2$ and $\Phi$ on $\R^3$ such that
$\phi(o)=p$, $\Phi(f(p))=\mb 0$ and
$$
(f_{3/2}:=)(u,v^2,v^3)=\Phi\circ f\circ\phi (u,v)\quad
\left(\mbox{resp.}\,\, 
(u, v^2, v^3\alpha(u,v))=\Phi\circ f\circ\phi(u,v)\,\right),
$$
where $\alpha(u,v)$ is a $C^r$-function.
Similarly,
a singular point $p\in U$ is called a {\it 
$5/2$-cuspidal edge point}
(resp.~a {\it fold singular point}) 
if there exist local $C^r$-diffeomorphisms
$\phi$ on $\R^2$ and $\Phi$ on $\R^3$ such that 
$\phi(o)=p$, $\Phi(f(p))=\mb 0$ and
$$
(f_{5/2}:=)(u,v^2,v^5)=\Phi\circ f\circ\phi (u,v)\quad
\left(\mbox{resp.}\,\, 
(u, v^2, 0)=\Phi\circ f\circ\phi(u,v)\,\right).
$$
Also, a singular point $p\in U$ is called a {\it 
cuspidal cross cap point} 
if there exist local $C^r$-diffeomorphisms
$\phi$ on $\R^2$ and $\Phi$ on $\R^3$ such that 
$\phi(o)=p$, $\Phi(f(p))=\mb 0$ and
$$
(f_{\ccr}:=)(u,v^2,uv^3)=\Phi\circ f\circ\phi(u,v).
$$
These singular points are
all generalized cuspidal edge points.

Let ${\mc G}^{r}_{3/2}(\R^2_o,\R^3)$
(resp. ${\mc G}^r(\R^2_o,\R^3)$)
be the set of germs of $C^r$-cuspidal edges 
(resp. generalized $C^r$-cuspidal edges)
$f(u,v)$ satisfying $f(o)=\mb 0$.
We fix $l>0$ and
consider an embedding (i.e. a simple regular space curve)
$$
\mb c:J\to \R^3\qquad (J:=[-l,l])
$$
such that $\mb c(0)=\mb 0$. 
We do not assume here that $u\mapsto \mb c(u)$
is the arc-length parametrization
(if necessary, we assume this in latter sections).
We denote by $C$ the image of 
$\mb c$.
Here, we ignore the orientation of $C$
and think of it as the singular set image (i.e. the image of the singular set)
of $f$.
We let ${\mc G}^{r}_{3/2}(\R^2_o,\R^3,C)$
(resp. ${\mc G}^r(\R^2_o,\R^3,C)$)
be the subset of ${\mc G}^{r}_{3/2}(\R^2_o,\R^3)$
(resp. ${\mc G}^r(\R^2_o,\R^3)$)
such that 
the singular set image of $f$ is contained in $C$
(we call $C$ the {\it edge} of $f$). 
Similarly, a subset of 
${\mc G}^{r}(\R^2_o,\R^3,C)$ denoted by
$$
{\mc G}^{r}_{\tccr}(\R^2_o,\R^3,C),\qquad
(\mbox{resp.}\,\,\, {\mc G}^{r}_{5/2}(\R^2_o,\R^3,C)\,)
$$
consisting of germs of cuspidal cross caps 
(resp. $5/2$-cuspidal edges) is also defined.

Throughout this paper, we assume
the curvature function $\kappa(u)$ of $\mb c(u)$ satisfies
\begin{equation}\label{eq:condK}
\kappa(u)>0\qquad (u\in J).
\end{equation}

\begin{figure}[bth]
\begin{center}
        \includegraphics[height=6.5cm]{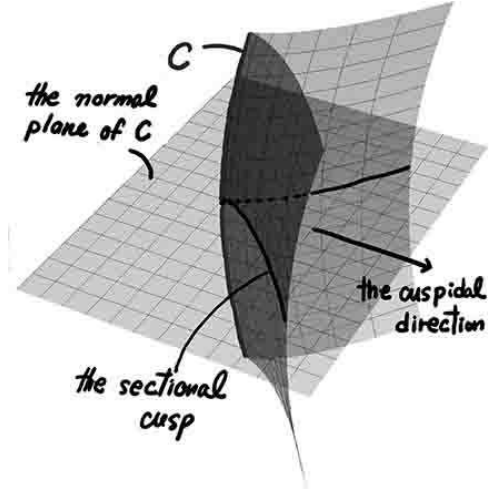}
\caption{A cuspidal edge and its sectional cusp}
\label{Fig:crossing1}
\end{center}
\end{figure}

Let $U$ be a neighborhood of $J\times \{0\}$ of $\R^2$
and $f:U\to \R^3$ a $C^r$-map consisting only of generalized
cuspidal edge points along $J\times \{0\}$
such that
\begin{equation}\label{eq:FC}
f(u,0)=\mb c(u)\qquad (u\in J).
\end{equation}
We denote by $\mc G^{r}(\R^2_J,\R^3,C)$ the set of such $f$
($f$ is called a {\it generalized cuspidal edge along $C$}).
Like as the case of map germs at $o$, the sets
$$
{\mc G}^{r}_{3/2}(\R^2_J,\R^3,C),\quad 
{\mc G}^{r}_{\tccr}(\R^2_J,\R^3,C),\quad
 {\mc G}^{r}_{5/2}(\R^2_J,\R^3,C)
$$
are also canonically defined.
For each point $P$ on the edge $C$, 
the plane $\Pi(P)$ passing through $P$
which is perpendicular to the curve $C$
is called the {\it normal plane} of $f$
at $P$. 
The section of the image of $f$ by the normal plane 
$\Pi(P)$ of $C$ at $P$ is a planar curve with 
a singular point at $P$.
We call this the {\it sectional cusp} of $f$ at $P$.
Moreover, we can find a tangent
vector $\mb v\in T_P\R^3$ at $P$, which points in
the tangential direction of the sectional cusp at $P$.
We call $\mb v$ the {\it cuspidal direction}
 (cf. \eqref{eq:fvv} and Figure \ref{Fig:crossing1}).
The angle $\theta_P$ of the cuspidal direction 
from the principal normal vector of $C$
at $P$ is called the {\it cuspidal angle}.

If we normalize the initial value 
$\theta_{\mb c(0)}\in (-\pi,\pi]$ 
at $\mb c(0)(=\mb 0)$,
then the cuspidal angle 
$$
\theta(u):=\theta_{\mb c(u)}\qquad (u\in J)
$$ 
at $\mb c(u)$
can be uniquely determined
as a $C^r$-function on $J$.
In \cite{MSUY, SUY}, the singular curvature 
$\kappa_s(u)$ 
and the limiting normal curvature 
$\kappa_\nu(u)$ 
along the edge $\mb c(u)$
are defined.
In our present situation,
they can be expressed as (cf. \cite[Remark 1.9]{F})
\begin{equation}\label{eq:sn2}
\kappa_s(u):=\kappa(u) \cos \theta(u),
\qquad \kappa_\nu(u):= \kappa(u) \sin \theta(u)\qquad (u\in J).
\end{equation}
By definition, 
$
\kappa(u)=\sqrt{\kappa_s(u)^2+\kappa_{\nu}(u)^2}
$
holds on $J$.
We say that $f\in \mc G^r(\R^2_o,\R^3,C)$ 
is {\it generic} at $o$   if
\begin{equation}\label{eq:kks}
|\kappa_s(0)|<\kappa(0).
\end{equation}
We denote by 
${\mc G}^r_*(\R^2_o,\R^3,C)$ the set of germs of
generic generalized $C^r$-cuspidal 
edges in ${\mc G}^r(\R^2_o,\R^3,C)$, and set
\begin{align}
\nonumber
{\mc G}^r_{*,3/2}(\R^2_o,\R^3,C)&:=
{\mc G}^r_*(\R^2_o,\R^3,C)\cap 
{\mc G}^r_{3/2}(\R^2_o,\R^3,C), \\
\label{eq:GGG}
{\mc G}^r_{*,\tccr}(\R^2_o,\R^3,C)&:=
{\mc G}^r_*(\R^2_o,\R^3,C)\cap 
{\mc G}^r_{\tccr}(\R^2_o,\R^3,C), \\
\nonumber
{\mc G}^r_{*,5/2}(\R^2_o,\R^3,C)&:=
{\mc G}^r_*(\R^2_o,\R^3,C)\cap 
{\mc G}^r_{5/2}(\R^2_o,\R^3,C).
\end{align}
On the other hand, for $f\in \mc G^r(\R^2_J,\R^3,C)$,
we consider the condition
\begin{equation}\label{eq:kks2}
|\kappa_s(u)|<\kappa(u)\qquad (u\in J),
\end{equation}
which implies that all singular points of $f$ along the
curve $C$ are generic.
We denote by
\begin{equation}\label{eq:admissble0}
\mc G^r_{*}(\R^2_J,\R^3,C) 
\end{equation}
the set of $f\in \mc G^r(\R^2_J,\R^3,C)$ satisfying \eqref{eq:kks2}.
Moreover, 
if 
\begin{equation}\label{eq:kks3}
\max_{u\in J}|\kappa_s(u)|<\min_{u\in J}\kappa(u) 
\end{equation}
holds, then $f$ is said to be {\it admissible}.
We denote by
\begin{equation}\label{eq:admissible}
\mc G^r_{**}(\R^2_J,\R^3,C) 
\end{equation}
the set of admissible $f\in \mc G^r(\R^2_J,\R^3,C)$.
Then by imitating \eqref{eq:GGG},
\begin{equation}\label{eq:admissible2}
\mc G^r_{*,3/2}(\R^2_J,\R^3,C),\qquad \mc G^r_{**,3/2}(\R^2_J,\R^3,C)
\end{equation}
are also defined. 
The following assertion is obvious:

\begin{lemma}\label{lem:Red}
Suppose that $f$ belongs to ${\mc G}^r_{3/2}(\R^2_o,\R^3,C)$ 
$($resp. ${\mc G}^r_{*,3/2}(\R^2_{o},\R^3,C)\,)$.
Then there exists $\epsilon(>0)$
such that $f$ is an element of
${\mc G}^r_{3/2}(\R^2_{J(\epsilon)},\R^3,C)$
$($resp. ${\mc G}^r_{**,3/2}(\R^2_{J(\epsilon)},\R^3,C))$,
where $J(\epsilon):=[-\epsilon,\epsilon]$.
\end{lemma}

Let $\op{O}(3)$ (resp.~$\op{SO}(3)$)
be the orthogonal group
(resp. the special orthogonal group) 
as the isometry group
(resp. the orientation preserving isometry group)
of $\R^3$ fixing the origin $\mb 0$.

\begin{defi}\label{def:distinct}
Suppose that $f_i$ ($i=1,2$) are generalized cuspidal edges
belonging to $\mc G^r(\R^2_o,\R^3,C)$ (resp. $\mc G^r(\R^2_J,\R^3,C)$).
Then the image of $f_1$ 
is said to have the {\it same image} as $f_2$ if 
there exists a neighborhood $U_i(\subset \R^2)$ of $o$ (\text{resp}. $J\times \{0\}$)
such that
$f_1(U_1)=f_2(U_2)$.
On the other hand,
$f_1$ is said to be {\it congruent} to $f_2$
if there exists an orthogonal matrix $T\in \op{O}(3)$
such that $T\circ f_1$ 
has the same image as $f_2$.
\end{defi}

We then define the following two equivalence relations:

\begin{defi}\label{def:ds-sym1}
For a given $f$ belonging to ${\mc G}^r(\R^2_o,\R^3,C)$ 
(resp. ${\mc G}^r(\R^2_J,\R^3,C)$),
we denote by $ds^2_f$ its first fundamental form.
A generalized cuspidal
edge $g$
belonging to ${\mc G}^r(\R^2_o,\R^3,C)$ (resp. ${\mc G}^r(\R^2_J,\R^3,C))$
is said to be {\it right equivalent} 
to $f$
if there exists a diffeomorphism $\phi$
defined on a neighborhood of $o$ (resp. $J\times \{0\}$)
in $\R^2$ such that $g=f\circ \phi$. 
\end{defi}

\begin{defi}\label{def:ds-sym2}
For a given generalized cuspidal
edge $f\in {\mc G}^r(\R^2_o,\R^3,C)$ (resp. ${\mc G}^r(\R^2_J,\R^3,C)$),
we denote by $ds^2_f$ its first fundamental form.
A generalized cuspidal
edge $g\in {\mc G}^r(\R^2_o,\R^3,C)$ (resp. ${\mc G}^r(\R^2_J,\R^3,C)$)
is said to be {\it isometric}
to $f$
if there exists a diffeomorphism $\phi$
defined on a neighborhood of $o$ (resp. $J\times \{0\}$)
in $\R^2$
such that $\phi^*ds^2_f=ds^2_g$.

In particular, we consider the case $f=g$.
If $\phi^*ds^2_f=ds^2_f$
and $\phi$ is not the identity map,
then $\phi$ is called
a {\it symmetry} of $ds^2_f$.
Moreover, if $\phi$ reverses 
the orientation of the singular curve of $f$,
then $\phi$ is said to be {\it effective}.
\end{defi}

\begin{remark}\label{Z}
A cuspidal edge $g\in \mathcal G^r_{3/2}(\R^2_o,\R^3,C)$ 
(resp. $\mathcal G^r_{3/2}(\R^2_J,\R^3,C)$)
has the same image as a given germ 
$f\in \mathcal G^r_{3/2}(\R^2_o,\R^3,C)$
(resp. $\mathcal G^r_{3/2}(\R^2_J,\R^3,C)$)
if and only if $g$ is right equivalent to $f$
(cf. \cite{KRSUY}).
\end{remark}

If two generalized cuspidal edges
$f,g\in {\mc G}^r(\R^2_o,\R^3,C)$ 
(resp. ${\mc G}^r(\R^2_J,\R^3,C)$)
are right equivalent,
then they are isometric each other.
However, the converse may not be true.
So we give the following:

\begin{defi}\label{def:IS}
For a given $f\in {\mc G}^r(\R^2_o,\R^3,C)$
(resp. ${\mc G}^r(\R^2_J,\R^3,C)$),
a generalized cuspidal
edge $g\in {\mc G}^r(\R^2_o,\R^3,C)$
(resp. ${\mc G}^r(\R^2_J,\R^3,C)$)
is called an {\it isomer} of $f$
(cf. \cite{NUY}) if it satisfies the 
following conditions;
\begin{enumerate}
\item $g$ is isometric to $f$, and
\item $g$ is not right equivalent to $f$.
\end{enumerate}
In this situation,
we say that $g$ is a {\it faithful isomer} of $f$ if
\begin{itemize}
\item there exists a local diffeomorphism $\phi$
such that $\phi^*ds^2_f=ds^2_g$, and
\item the orientations of $C$ induced
by $u\mapsto f\circ \phi(u,0)$ 
and $u\mapsto g(u,0)$ 
are compatible with respect to the one
induced by
$u\mapsto f(u,0)$.
\end{itemize}
\end{defi}

In \cite[Corollary D]{NUY}, it was shown the existence
of an involution 
\begin{equation}\label{eq:I}
{\mc G}^\omega_{*,3/2}(\R^2_o,\R^3,C)\ni f \mapsto \check 
f\in {\mc G}^\omega_{*,3/2}(\R^2_o,\R^3,C).
\end{equation}
To construct $\check f$,
we need to apply the so-called Cauchy-Kowalevski theorem
on partial differential equations of real analytic category 
(cf. Theorem \ref{thm:S1}).
Here, $\check f$ is called the 
{\it isometric dual} of $f$, which 
satisfies the following properties:
\begin{enumerate}
\item[(i)] The first fundamental form of $\check f$
coincides with that of $f$. 
\item[(ii)] The map $\check f$ 
is a faithful isomer of $f$.
\item[(iii)]
If $\theta(P)$ is the cuspidal angle of $f$ at 
$P(\in C)$,
then $-\theta(P)$ is the cuspidal angle of $\check f$ 
at $P$.
\end{enumerate}
In \cite{NUY}, a necessary and sufficient condition
for a given positive semi-definite metric to be 
realized as the first fundamental form of a cuspidal edge 
along $C$ is given.
In this paper, 
we first prove the following using the method given in \cite{NUY}:

\medskip
\noindent
{\bf Theorem I.} \label{f1}
{\it There exists an involution $($called the first involution$)$
\begin{equation}\label{eq:IcA}
\mc I_C:{\mc G}^\omega_*(\R^2_J,\R^3,C)\ni f \mapsto \check f
\in {\mc G}^\omega_*(\R^2_J,\R^3,C)
\end{equation}
defined on ${\mc G}^\omega_*(\R^2_J,\R^3,C)$ $($cf.~\eqref{eq:admissble0}$)$
satisfying the properties 
{\rm (i)}, {\rm (ii)} and {\rm (iii)} above. 
Moreover,
regarding $f$ and $\check f$ as
map germs at $o$ $($cf. Lemma \ref{lem:Red}$)$, $\mc I_C$
induces a map 
\begin{equation}\label{eq:Io}
\mc I_o:{\mc G}^\omega_*(\R^2_o,\R^3,C)\ni f \mapsto \check f
\in {\mc G}^\omega_*(\R^2_o,\R^3,C),
\end{equation}
which gives a generalization of the map as
 in \eqref{eq:I}.}

\medskip
The existence of the map $\mc I_o$ follows also from
\cite[Theorem B]{HNUY}, since $\check f$ is strongly congruent to $f$
in the sense of \cite[Definition 3]{HNUY}.
However, the existence of the map $\mc I_C$ itself
does not follow from \cite{HNUY}, 
since $\check f$ given in Theorem~I is not a map germ at $o$, but a
map germ along the curve $C$.
Some variants of this result 
for germs of swallowtails and cuspidal cross caps
were given in \cite[Theorem B]{HNUY} using a method different from \cite{NUY}. 
(For swallowtails,
the duality corresponding to the above properties (i), (ii) and (iii) are not
obtained, see item (4) below.)
The authors find Theorem I to be 
suggestive of the following geometric problems:
\begin{enumerate}
\item How many right equivalence classes 
of isomers of $f$ exist other than $\check f$?
\item When are isomers non-congruent to each other? 
\item The existence of the isometric dual can be proved by 
applying the Cauchy-Kowalevski theorem.  So we need 
to assume that the given generalized cuspidal edges are real analytic.
It is then natural to ask if 
one can find a new method for constructing the isometric dual
in the $C^\infty$-differentiable category.
\item Can one extend isometric duality
to a much wider class, say, for swallowtails? 
\end{enumerate}
In this paper, we show the following:
\begin{itemize}
\item For a given generalized cuspidal edge 
$f\in \mc G^\omega_{**}(\R^2_J,\R^3,C)$,
there exists a unique generalized
cuspidal edge $f_*\in 
\mc G^\omega_{**}(\R^2_J,\R^3,C)$ (called the {\it inverse} of $f$)
having the same first fundamental form as $f$
along the space curve $\mb c(-u)$
whose cuspidal angle has the same sign as
that of $f$. Moreover, any isomers of $f$ are right equivalent to one of 
$\{f,\check f, f_*,\check f_*\}$ (see Theorem~{\rm II}),
where $\check f_*:=\mc I_C(f_*)$ is called the {\it inverse dual} of $f$.
\item The four maps $f,\,\,\check f,\,\,f_*,\,\,\check f_*$
are non-congruent in general. 
Moreover, the right equivalence classes and congruence classes 
of these four surfaces are determined 
in terms of the properties of $C$ 
and $ds^2_f$ (cf. Theorems~{\rm III} and {\rm IV}).
\item Suppose that the image of a $C^\infty$-differentiable 
cuspidal edge $f$ is invariant under a non-trivial
symmetry $T\in \op{SO}(3)$ (cf.~Definition \ref{def:S2}) of $\R^3$.
Then explicit construction of $\check f$
without use of the Cauchy-Kowalevski theorem is given
(see Example \ref{thm5}).
\end{itemize}
About the last question (4), 
the authors 
do not know whether the isomers of a given swallowtail 
will exist in general,
since the method given in this
paper does not apply directly.
So it left here as an open problem. 
(A possible isometric deformations of swallowtails are discussed in
authors' previous work \cite{HNUY}.)

The paper is organized as follows:
In Section 1, we explain our main results. 
In Section 2, we review the definition and properties of
Kossowski metrics.
In Section~3, we prove Theorem~{\rm I}
as a modification of the proof of \cite{NUY}.
In Section~4, we recall a representation formula for
generalized cuspidal edges given in Fukui \cite{F}, and prove 
Theorem~{\rm II}.  
In Section~5, we investigate the properties of
generic cuspidal edges with symmetries.
Moreover, we prove Theorems {\rm III} and {\rm IV}. 
Several examples are given in Section~6.
Finally, in the appendix, a representation formula
for generalized cusps in the Euclidean plane is given.

\medskip
\section{Results}

Let $ds^2$ be a $C^r$-differentiable positive semi-definite metric
on a $C^r$-differentiable $2$-manifold $M^2$.
A point $o\in M^2$ is called a {\it regular point}
of $ds^2$ if it is positive definite at $o$,
and is called a {\it singular point} (or a {\it semi-definite point})
if $ds^2$ is not positive definite at $o$.
Kossowski \cite{K} defined a certain kind of positive semi-definite
metrics called \lq\lq Kossowski metrics\rq\rq\ (cf. Section 2).
We let $ds^2$ be such a metric. Then for each singular point $o\in M^2$,
there exists a regular curve $\gamma:(-\epsilon,\epsilon)\to M^2$
such that $\gamma(0)=o$ and $\gamma$ parametrizes the singular set of $ds^2$
near $o$. Such a curve is called the {\it singular curve} of $ds^2$ near $o$.
In this situation, if $ds^2(\gamma'(0),\gamma'(0))$ does not vanish, 
then
we say that \lq\lq $ds^2$ is of type {\rm I} at $o$''.
The first fundamental forms (i.e. the induced metrics)
of germs of generalized 
cuspidal edges are Kossowski metrics of  
type {\rm I} (cf. Proposition \ref{lem:GK}).   

Setting $M^2:=(\R^2; u,v)$, we denote by $\mc K^{r}_{\rm I}(\R^2_o)$
the set of germs of $C^r$-Kossowski metrics 
of type {\rm I} at $o:=(0,0)$.
We fix such a $ds^2\in \mc K^r_{\rm I}(\R^2_o)$.
Then the metric is expressed as
$$
ds^2=Edu^2+2Fdudv+Gdv^2,
$$
and there exists a $C^r$-function $\lambda$
such that
$
EG-F^2=\lambda^2.
$
Let $K$ be the Gaussian curvature of $ds^2$
defined at points where $ds^2$ is positive definite.
Then 
\begin{equation}\label{eq:hK}
\hat K:=\lambda K
\end{equation}
can be considered as a $C^r$-differentiable 
function defined on a neighborhood $U(\subset \R^2)$
of $o$ (cf. \cite{MSUY,HNUY}).
If $\hat K$ vanishes 
(resp. does not vanish) at a singular
point $q\in U$ of $ds^2$, 
then $ds^2$ is said to be {\it parabolic}
(resp.~{\it non-parabolic}) at $q$  (see Definition~\ref{def:np}).
We denote by $\mc K^r_*(\R^2_o)$ (resp. $\mc K^r_p(\R^2_o)$)
the set of germs of non-parabolic (resp. parabolic)
$C^r$-Kossowski metrics of type {\rm I} at $o$.
The subset of $\mc K^r_p(\R^2_o)$ defined by
\begin{align*}
\mc K^r_{p,*}(\R^2_o)&:=\{ds^2\in \mc K^r_{p}(\R^2_o)\,;\, 
\hat K'(o)\ne 0\} \\
&\left(=\{ds^2\in \mc K^r_{\rm I}(\R^2_o)\,;\, 
\hat K(o)= 0,\,\, \hat K'(o)\ne 0\}\right)
\end{align*}
plays an important role in this paper, 
where $\hat K'=\partial \hat K/\partial u$.
Metrics belonging to $\mc K^r_{p,*}(\R^2_o)$ are called
{\it p-generic}.
On the other hand, if $\hat K$ vanishes identically along
the singular curve of $ds^2\in \mc K^r_{\rm I}(\R^2_o)$,
we call $ds^2$ an {\it asymptotic} Kossowski metric of type {\rm I}.
We let $\mc K^r_{a}(\R^2_o)$ be the set of
germs of such metrics.
This terminology comes from the following two facts:
\begin{itemize}
\item 
for a regular surface, a direction where the
normal curvature vanishes
is called an asymptotic direction, and
\item the induced metric of a
cuspidal edge whose limiting normal curvature $\kappa_\nu$ vanishes
identically along its singular set belongs to $\mc K^r_{a}(\R^2_o)$.
(Such a cuspidal edge is called an {\it asymptotic cuspidal edge},
see Proposition~\ref{prop:4b}.)
\end{itemize}
By definition, we have
\begin{align*}
&\mc K^r_{*}(\R^2_o)\cap \mc K^r_{p}(\R^2_o)=\emptyset,\qquad
\mc K^r_{*}(\R^2_o)\cup \mc K^r_{p}(\R^2_o)=\mc K^r_{\rm I}(\R^2_o), \\
&\mc K^r_{a}(\R^2_o)\subset \mc K^r_{p}(\R^2_o)\subset 
\mc K^r_{\rm I}(\R^2_o).
\end{align*}
For $ds^2\in \mc K^r_{a}(\R^2_o)$,
the Gaussian curvature $K$ can be extended on a neighborhood of
$o$ 
as a $C^r$-differentiable function.
Let $\eta\in T_o\R^2$ be the null vector at the singular point $o$ 
of the asymptotic Kossowski metric $ds^2$.
If
\begin{equation}\label{eq:Kv}
dK(\eta)(o)\ne 0,
\end{equation}
then $ds^2$ is said to be {\it a-generic}, 
and we denote by $\mc K^r_{a,*}(\R^2_o)(\subset \mc K^r_{a}(\R^2_o))$ 
the set of germs of a-generic asymptotic $C^r$-Kossowski metrics.
Considering the first fundamental 
form $ds^2_f$ of $f$, we can define a map 
\begin{equation}\label{eq:Jo}
\mc J_o:{\mc G}^r_*(\R^2_o,\R^3,C)
\ni f \mapsto ds^2_f \in \mc K^r_{\rm I}(\R^2_o).
\end{equation}

\medskip
\noindent
{\bf Theorem {\rm II}.}
{\it 
There exists an involution $($called the second involution$)$
$$
\mc I^*_C:{\mc G}^\omega_{**}(\R^2_J,\R^3,C)\ni f \mapsto f_*
\in {\mc G}^\omega_{**}(\R^2_J,\R^3,C)
$$
defined on ${\mc G}^\omega_{**}(\R^2_J,\R^3,C)$
 $($cf.\ \eqref{eq:admissible2}$)$
satisfying the following properties: 
\begin{enumerate}
\item $f_*$ has the same first fundamental form as $f$, and is a
non-faithful isomer of $f$,
\item $\mc I^*_C\circ \mc I_C=\mc I_C\circ \mc I^*_C$, 
where $\mc I_C$ is the first involution
as in Theorem {\rm I}.
\item 
Regarding $f$ and $f_*$ as
map germs at $o$ $($cf. Lemma \ref{lem:Red}$)$,  $\mc I^*_C$ canonically induces a map
\begin{equation}\label{eq:I*}
\mc I^*_o:{\mc G}^\omega_*(\R^2_o,\R^3,C)\ni f \mapsto f_*
\in {\mc G}^\omega_*(\R^2_o,\R^3,C)
\end{equation}
such that $\mc J_o\circ \mc I^*_o=\mc J_o$
and $\mc I^*_o\circ \mc I_o=\mc I_o\circ \mc I^*_o$.
\item Suppose that  $g$ belongs to ${\mc G}^\omega_{*}(\R^2_o,\R^3,C)$
$($resp. ${\mc G}^\omega_{**}(\R^2_J,\R^3,C))$.
If the first fundamental form of $g$
is isometric to that of $f$, then $g$
is right equivalent to 
one of $f,\check f,f_*$ and $\check f_*$.
\end{enumerate}}%

\medskip
Recently,
Fukui \cite{F} gave a representation formula for
generalized cuspidal edges along their edges in $\R^3$. 
(In \cite{F}, a similar formula 
for swallowtails is also given, 
although it is not applied in this paper.)
We denote by $C^r(\R_o)$ (resp.~$C^r(\R^2_o)$)
the set of $C^r$-function germs 
at the origin of $\R$ (resp.~$\R^2$).
We fix a generalized cuspidal edge $f\in {\mc G}^r(\R^2_o,\R^3,C)$
arbitrarily. The sectional cusp of $f$ at $\mb c(u)$ induces 
a function $\mu(u,t)\in C^r(\R^2_o)$ which 
is called the \lq\lq extended half-cuspidal curvature function''
giving the normalized curvature function of the sectional cusp at $\mb c(u)$
(see the appendix).
The value
\begin{equation}\label{eq:m0}
\kappa_c(u):=\frac{\mu(u,0)}2
\end{equation}
coincides with the cuspidal curvature at
the singular point of the sectional
cusp, and so it is called the {\it cuspidal curvature function} 
of $f$ (cf. \cite{MSUY}).
In Section 4, we give a Bj\"orling-type representation formula 
for cuspidal edges (cf. Proposition~\ref{prop:Bj}),
 which is a modification of the formula given in Fukui \cite{F}. 
(In fact, Fukui \cite{F} expressed the sectional cusp
as a pair of functions, but did not use the function $\mu$.)
Fukui \cite{F} explained several geometric invariants of
cuspidal edges in terms of $\kappa_s,\kappa_\nu$ and $\theta$.
In Section 4, using several properties of modified Fukui's formula
together with the proof of Theorem {\rm I},
we reprove the following assertion  
which determine the images of the maps
$\mc I_o$ and $\mc J_o$ (the assertions for the map $\mc I^*_o$
are not given in \cite{NUY,HNUY,HS}):

\begin{fact}\label{f2}
{\it The maps $\mc I_o$, $\mc I^*_o$ and $\mc  J_o$ $($cf. \eqref{eq:Io},
\eqref{eq:Jo} and \eqref{eq:I*}$)$
 satisfy the followings: 
\begin{enumerate}
\item[(1)] 
These two maps $\mc I_o$ and $\mc I^*_o$ are involutions on 
${\mc G}^\omega_{*,3/2}(\R^2_o,\R^3,C)$,
and $\mc J_o$ maps ${\mc G}^\omega_{*,3/2}(\R^2_o,\R^3,C)$ onto 
$\mc K^\omega_*(\R^2_o)$
$($cf.~\cite[Theorem 12]{NUY}$)$.
\item[(2)] 
The two maps $\mc I_o$ and $\mc I^*_o$ are involutions on
${\mc G}^\omega_{*,\tccr}(\R^2_o,\R^3,C)$,
and $\mc J_o$ maps ${\mc G}^\omega_{*,\tccr}(\R^2_o,\R^3,C)$
onto $\mc K^\omega_{p,*}(\R^2_o)$
$($cf.~\cite[Theorem A]{HNUY}$)$.
\item[(3)] 
The two maps $\mc I_o$ and $\mc I^*_o$ are involutions on
${\mc G}^\omega_{*,5/2}(\R^2_o,\R^3,C)$,
and $\mc J_o$ maps ${\mc G}^\omega_{*,5/2}(\R^2_o,\R^3,C)$ onto
$\mc K^\omega_{a,*}(\R^2_o)$
$($cf. \cite[Theorem 5.6]{HS}$)$. 
\end{enumerate}}
\end{fact}

We may assume that the origin $\mb 0$ is the midpoint of 
$C$, and give here the following terminologies:

\begin{defi}\label{def:S2}
The curve $C$ admits a {\it  symmetry} at $\vect{0}$
if there exists $T\in \op{O}(3)$ such that $T(C)=C$
and $T$ is not the identity.
Moreover, $T$ is said to be {\it trivial} if
$T(P)=P$ for all $P\in C$.
A symmetry of $C$ which is not trivial is called
a {\it non-trivial symmetry}.
(Obviously, each non-trivial symmetry reverses
the orientation of $C$.)
A non-trivial symmetry is called {\it positive} 
(resp. {\it negative}) 
if $T\in \op{SO}(3)$
(resp. $T\in \op{O}(3)\setminus \op{SO}(3)$).
\end{defi}

If $C$ lies in a plane, then there exists
a reflection $S\in \op{O}(3)$ with respect to the plane.
Then $S$ is a trivial symmetry of $C$.
We prove the following assertion.

\medskip
\noindent
{\bf Theorem {\rm III}.}
{\it
Let $f\in {\mc G}^\omega_{**, 3/2}(\R^2_J,\R^3,C)$, that is,
$f$ is admissible. 
Then the number of the right equivalence classes of
$f$, $\check f$, $f_*$ and
$\check f_*$ is four if and only if $ds^2_f$ has no 
symmetries
$($cf. Definition \ref{def:ds-sym2}$)$.}

\rm
\medskip
Moreover, we can prove the following:

\medskip
\noindent
{\bf Theorem {\rm IV}.}
{\it Let $f\in {\mc G}^\omega_{**,3/2}(\R^2_J,\R^3,C)$.
Then the number $N_f$ of 
the congruence classes of the images of
$f,\check f$, $f_*$ and $\check f_*$
satisfies the following properties: 
\begin{enumerate}
\item If $C$ has no non-trivial symmetries, and also $ds^2_f$ has no symmetries,
then $N_f=4$,
\item if not the case in {\rm (1)}, it holds that $N_f\le 2$,
\item $N_f=1$ if and only if 
\begin{enumerate}
\item $C$ lies in a plane and has a non-trivial symmetry, 
\item $C$ lies in a plane and
$ds^2_f$ has a symmetry, or
\item $C$ has a positive symmetry
and $ds^2_f$ also has a symmetry.
\end{enumerate}
\end{enumerate}
}

\section{Kossowski metrics}\label{sec1}

In this section, we quickly review several fundamental
properties of Kossowski metrics.

\begin{defi}
Let $p$ be a singular point of a given positive semi-definite
 metric $ds^2$ on $M^2$.
Then a non-zero tangent vector $\vect{v}\in T_pM^2$
is called a {\it null vector} if
\begin{equation}\label{eq:vv}
ds^2(\vect{v},\vect{v})=0.
\end{equation}
Moreover, a local coordinate neighborhood $(U;u,v)$ 
is called  {\it adjusted} at $p\in U$ if
$\partial_v:=\partial/\partial v$ gives a null vector
of $ds^2$ at $p$.
\end{defi}
  
It can be easily checked that
\eqref{eq:vv} implies that
 $ds^2(\vect{v},\vect{w})=0$
for all $\vect{w}\in T_pM^2$.
If $(U;u,v)$ is a local coordinate neighborhood 
adjusted at $p\in U$, then 
$
F(p)=G(p)=0
$ holds, where
\begin{equation}\label{eq:ds20}
  ds^2=E\,du^2+2F\,du\,dv+G\,dv^2.
\end{equation}

\begin{defi}\label{def:K}
A singular point $p\in M^2$ 
of a $C^r$-differentiable positive semi-definite metric
$ds^2$ on $M^2$ is called {\it K-admissible}
if there exists a local coordinate
neighborhood $(U;u,v)$ adjusted at $p$ satisfying
\begin{equation}\label{eq:d-flat}
  E_v(p)=2F_u(p),\qquad
  G_u(p)=G_v(p)=0,
\end{equation}
where $E,F,G$ are the $C^r$-functions on $U$
given in \eqref{eq:ds20}.
\end{defi}

If $ds^2_f$ is the induced metric of a $C^r$-map
$f:U\to \R^3$ and $f_v(p)=\vect{0}$, then 
\eqref{eq:d-flat} is satisfied automatically
(cf. Proposition \ref{lem:GK}).
The property \eqref{eq:d-flat} does not 
depend on the choice of a local coordinate system adjusted at 
$p$, as shown in \cite{K} and \cite[Proposition 2.7]{HHNSUY}.
In fact, a coordinate-free treatment for the K-admissibility
of singular points is given in \cite{K} and \cite[Definition 2.3]{HHNSUY}.

\begin{defi}\label{def:K2}
A positive semi-definite $C^r$-differentiable metric $ds^2$
is 
called a {\it Kossowski metric} if
each singular point $p\in M^2$
of $ds^2$ is {\rm K}-admissible
and there exists
a $C^r$-function $\lambda(u,v)$ 
defined on a local coordinate
neighborhood $(U;u,v)$ of $p$ 
such that
\begin{align}
\label{eq:k1}
&  
EG-F^2=\lambda^2 \qquad (\mbox{on $U$}), 
\\ \label{eq:k2} 
& 
(\lambda_u(p),\lambda_v(p))\ne (0,0),
\end{align}
where $E,F,G$ are $C^r$-functions on $U$
given in \eqref{eq:ds20}.
\end{defi}

The above function $\lambda$ is determined up to
$\pm$-ambiguity (see \cite[Proposition 3]{HNUY}).
We call such a $\lambda$ the {\it signed area density function}
of $ds^2$ with respect to the local coordinate 
neighborhood $(U;u,v)$.
The following fact is known (cf. \cite{K, SUY}).

\begin{fact}\label{lem:dA}
Let $ds^2$ be a $C^r$-differentiable 
Kossowski metric defined on a
domain $U$ of the $uv$-plane.
Then the $2$-form
$
d\hat A:=\lambda du\wedge dv
$
on $U$ is defined independently of the choice of 
adjusted local coordinates
$(u,v)$. 
\end{fact}

We call $d\hat A$ the {\it signed area form} of $ds^2$. 
Let $K$ be the Gaussian curvature 
defined on the complement of the singular set
of $ds^2$. 

\begin{fact}[{\cite{K} and \cite[Theorem 2.15]{HHNSUY}}]\label{lem:KdA}
The $2$-form
$
\Omega:=K d\hat A
$
can be extended as a $C^r$-differential form
on $U$. 
\end{fact}

\begin{defi}\label{def:np}
We call $\Omega$ the {\it Euler form} of $ds^2$. 
If $\Omega$ vanishes (resp. does not vanish)
at a singular point $p\in U$ of $ds^2$,
then $p$ is called a {\it parabolic point}
(resp. {\it non-parabolic point}).
\end{defi}

The following fact is also known (cf. \cite{K, HHNSUY, HNUY}).

\begin{fact}\label{lem:null-add}
Let $p$ be a singular point of a Kossowski metric
$ds^2$. Then the null space
$($i.e. the subspace generated by null vectors at $p)$ 
of $ds^2$ is $1$-dimensional.\end{fact}

By applying the implicit function theorem
for $\lambda$ (cf. \eqref{eq:k2}),
there exists a regular curve
$\gamma(t)$ $(|t|<\varepsilon)$ in the $uv$-plane 
(called the {\it singular curve})
parametrizing the singular set of $ds^2$
such that $\gamma(0)=p$.
Then there exists a $C^r$-differentiable non-zero vector field $\eta(t)$
along $\gamma(t)$ which points in the null direction of
the metric $ds^2$.
We call $\eta(t)$ a {\it null vector field} along 
the singular curve $\gamma(t)$.

\begin{defi}[{\cite{HHNSUY}}]
\label{def:a2a3add}
A singular point $p\in M^2$ of a
Kossowski metric $ds^2$ is said to be
of {\it type I}
or an {\it $A_2$ point}
if the derivative $\gamma'(0)$
of the singular curve at $p$
(called the {\em singular direction} at $\gamma(t)$) is linearly independent
of the null vector $\eta(0)$.
Moreover, $ds^2$ is called  
of {\it type I} if all of the singular points of $ds^2$
are of type I.
\end{defi}

\section{Generalized cuspidal edges}\label{sec-PI}

Fix a bounded closed interval $J(\subset \R)$ and
consider a $C^r$-embedding $\mb c:J\to \R^3$
with arc-length parameter.
We assume that the curvature function $\kappa(u)$
of $\mb c(u)$ is positive everywhere.
We fix a $C^r$-map $\tilde f:\tilde U\to \R^3$
defined on
a domain $\tilde U$ in the $xy$-plane $\R^2$ 
containing $J_1\times \{0\}$
such that
each point of $J_1\times \{0\}$
is a generalized cuspidal edge point
and 
$$
\tilde f(J_1\times \{0\})=C \qquad (C:=\mb c(J)),
$$
where $J_1$ is a bounded closed interval
in $\R$. Such an $\tilde f$ is called
a {\it generalized cuspidal edge along $C$.}
For such an $\tilde f$,
there exists a diffeomorphism
$$
\phi:U\ni (u,v)\mapsto (x(u,v),y(u,v))\in \phi(U)(\subset \tilde U)
$$
such that 
\begin{equation}\label{eq:F0}
f(u,v):=\tilde f(x(u,v),y(u,v))
\end{equation}
satisfies
\begin{equation}\label{eq:F}
f(u,v)=\mb c(u)+\frac{v^2}2 \hat \xi(u,v),
\end{equation}
where
$
\hat \xi(u,0)
$
gives a vector field along $\mb c$ which is
linearly independent of $\mb c'(u)$.

\begin{proposition}\label{lem:GK}
The induced metrics of $C^r$-differentiable 
generalized cuspidal edges 
are $C^r$-differentiable
Kossowski metrics whose singular points are
of type I.
\end{proposition}

\begin{proof}
Let $f$ be a generalized cuspidal edge as in \eqref{eq:F},
and let
$
ds^2_f=Edu^2+2Fdudv+Gdv^2
$
be the first fundamental form of $f$.
Then
$$
E=f_u\cdot f_u,\quad F=f_u\cdot f_v,\quad G:=f_v\cdot f_v 
$$
hold, where \lq\lq$\cdot$\rq\rq\ is the inner product of $\R^3$. 
Since $f_v(u,0)=\mb 0$, one can easily check \eqref{eq:d-flat}.
By \eqref{eq:F}, we have
$$
EG-F^2=|f_u\times f_v|^2
=v^2\left |
\left(\mb c'+\frac{v^2}2 \hat \xi_u\right)
\times \left(\hat \xi+\frac{v}2 \hat \xi_v\right)\right|^2,
$$
where $\times$ denotes the cross product in $\R^3$.
Since two vectors $\mb c'(u),\,\,\hat \xi(u,0)$ are linearly independent,
the function $\lambda$ on $U$ given by
\begin{equation}\label{eq:l0}
\lambda:=v \lambda_0,\qquad
\lambda_0:=\left|
\left(\mb c'+\frac{v^2}2 \hat \xi_u\right)
\times \left(\hat \xi+\frac{v}2 \hat \xi_v\right)
\right|
\end{equation}
is $C^r$-differentiable and $\lambda_0(u,0)\ne 0$. 
Moreover, $\lambda^2$ coincides with $EG-F^2$. 
Since $\lambda_v\ne 0$, $ds^2_f$ is a Kossowski metric.
Since $f_v(u,0)=\mb 0$, $\partial_v:=\partial/\partial v$ gives
the null-direction, which is linearly independent
of the singular direction $\partial_u$.
So all singular points of $ds^2_f$ are of type {\rm I}.
\end{proof}

Let $ds^2_f$ be the induced metric of $C^r$-differentiable 
generalized cuspidal edge $f\in \mc G^r(\R^2_J,\R^3,C)$.  
We set
$\hat K(:=\lambda K)$ {\rm (cf. \eqref{eq:hK})},
where $K$ is the Gaussian curvature of $ds^2_f$ defined
at points where $ds^2_f$ is positive definite.
As mentioned in the introduction, $\hat K$
can be extended as a $C^r$-function on $U$.
Moreover, $\check K:=v K$ also can be considered 
as a $C^r$-function on $U$ (cf. \cite{MSUY,HNUY}).

\begin{corollary}\label{cor:crr}
The following assertions hold:
\begin{enumerate}
\item $\hat K(u,0)\ne 0$  if and only if $\check K(u,0)\ne 0$, and
\item    
$\hat K_u(u,0)\ne 0$  if and only if $\check K_u(u,0)\ne 0$,
under the assumption $\hat K(u,0)=0$.
\end{enumerate}
\end{corollary}

\begin{proof}
By \eqref{eq:l0}, we have the expression
$\lambda=v \lambda_0$, where $\lambda_0(u,0)\ne 0$.
So if we set $\check K=v K$, then 
$\hat K=\lambda_0 \check K$,
and 
$
\hat K(u,0)=\lambda_0(u,0) \check K(u,0)
$
hold, and so the first assertion is obvious.
Differentiating $\hat K=\lambda_0 \check K$, we have
$$
\hat K_u=(\lambda_0)_u \check K+\lambda_0 \check K_u.
$$
Since $\hat K(u,0)=0$ implies $\check K(u,0)=0$, 
we have
$
\hat K_u(u,0)=\lambda_0(u,0) \check K_u(u,0),
$
proving the second assertion.
\end{proof}

\begin{remark}
For a generalized cuspidal edge $f$,
$$
\nu(u,v):=
\frac{(2\mb c'(u)+v^2\hat \xi_u(u,v))\times (2\hat \xi(u,v)+v \hat \xi_v(u,v))}
{|(2\mb c'(u)+v^2 \hat \xi_u(u,v))\times (2\hat \xi(u,v)+v \hat \xi_v(u,v)) |}
$$
gives a $C^r$-differentiable unit normal vector field on $U$.
So $f$ is a frontal map.
\end{remark}

\begin{defi}\label{def:adapted}
A parametrization $(u,v)$ of $f\in \mc G^r(\R^2_J,\R^3,C)$
is called an {\it adapted coordinate system} (cf. \cite[Definition 3.7]{MSUY}) 
if
\begin{enumerate}
\item $f_v(u,0)=\mb 0$ and $|f_u(u,0)|=|f_{vv}(u,0)|=1$ along the $u$-axis, 
\item $f_{vv}(u,0)$  is perpendicular to $f_u(u,0)$.
\end{enumerate}
\end{defi}

To show the existence of an adapted coordinate system, we prepare
the following under the assumption that the curve $\mb c(u)$
is real analytic:

\begin{lemma}[{\cite[Proposition 6]{HNUY}}]\label{fact:coord}
Let $ds^2$ be a $C^\omega$-differentiable 
Kossowski metric defined on 
an open subset $U(\subset \R^2)$.
Suppose that $\gamma:J\to U$
is a real analytic singular curve with respect to $ds^2$ such that 
\begin{equation}\label{eq:3-5}
ds^2(\gamma'(t),\gamma'(t))>0 \qquad (t\in J).
\end{equation}
Then, for each $t_0\in J$,
there exists a $C^\omega$-differentiable 
local coordinate system $(V;u,v)$ 
containing $(t_0,0)$ such that $V\subset U$ and 
the coefficients $E,F,G$ of the first fundamental 
form $ds^2=Edu^2+2Fdudv+Gdv^2$ satisfy
the following three conditions:
\begin{enumerate}
\item $\gamma(u)=(u,0)$, 
$E(u,0)=1$ and $E_v(u,0)=0$ hold along the $u$-axis,
\item $F(u,v)=0$ on $V$, and
\item there exists a $C^\omega$-function $G_0$ 
      defined on $V$ such that
      $G(u,v)=v^2 G_0(u,v)/2$ and $G_0(u,0)=2$.
\end{enumerate}
\end{lemma}
 
\begin{proof}
Applying \cite[Proposition 6]{HNUY} at the point $(t_0,0)$ on a
singular curve of $ds^2$, we obtain the desired local coordinate system.
\end{proof}

\begin{corollary}
For each generalized cuspidal edge $f\in \mc G^\omega(\R^2_J,\R^3,C)$ along $C$
and for each singular point $p$ of $f$,
there exists a local coordinate neighborhood $(V;u,v)$ of $p$
such that the restriction $f|_V$ of $f$
is parametrized by an adapted coordinate system.
\end{corollary}

\begin{proof}
We let $ds^2_f$ be the first fundamental form of $f(x,y)$.
By Lemma \ref{fact:coord}, we obtain a parameter change
$
(x,y) \mapsto (u(x,y),v(x,y))
$
on a neighborhood of $p$
such that the new parameter $(u,v)$ of
$f(u,v)$ defined by \eqref{eq:F0}
satisfies (1)-(3) of 
Lemma \ref{fact:coord} for the first fundamental form $ds^2_f$ of $f$.
Then we can show that this new coordinate system $(u,v)$
is the desired one:
Since the $u$-axis is the singular set of $ds^2_f$, we have
$
f_v(u,0)=\mb 0.
$
On the other hand,
$
f_u(u,0)\cdot f_u(u,0)=E(u,0)=1
$
and
\begin{equation}\label{eq:fvvu}
f_{vv}(u,0)\cdot f_u(u,0)=\left.\frac{\partial F(u,v)}{\partial v}\right |_{v=0}=0.
\end{equation}
Finally, we have
$$
f_{vv}(u,0)\cdot f_{vv}(u,0)
=\frac12 \left.\frac{\partial^2 G(u,v)}{\partial v^2}\right |_{v=0}
=\frac{G_0(u,0)}{2}=1,
$$
proving the assertion.
\end{proof}

From now on, we assume that $f(u,v)$ is parametrized by
the local coordinate system as in 
Definition \ref{def:adapted}.
Then $u$ is the arc-length parameter of the edge
$
\mb c(u):=f(u,0).
$
In this section, we assume that the curvature 
function $\kappa(u)$
of $\mb c(u)$ is positive for each $u$. 
Then the torsion function $\tau(u)$ is well-defined.
We can take the unit tangent vector
$
\mb e(u):=\mb c'(u)
$
(${}'=d/du$),
and the unit principal normal vector $\mb n(u)$
satisfying
$
\mb c''(u)=\kappa(u)\mb n(u).
$
We set
$$
\mb b(u):=\mb e(u)\times \mb n(u),
$$
which is the binormal vector of $\mb c(u)$.
Since $f_{vv}(u,0)$ is perpendicular to $\mb e(u)$,
we can write
\begin{equation}\label{eq:fvv}
f_{vv}(u,0)=\cos \theta(u) \mb n(u)-\sin \theta(u) \mb b(u),
\end{equation}
which is called the {\it cuspidal direction}.
As defined in the introduction, 
\begin{itemize}
\item 
the plane $\Pi(\mb c(u))$ passing through $\mb c(u)$
spanned by $\mb n(u)$ and $\mb b(u)$ is 
the normal plane of the space curve $\mb c(u)$,
\item the section of the image of $f$ by
$\Pi(\mb c(u))$ is a plane curve, which is
called the {\it sectional cusp} at $\mb c(u)$, and
\item 
the vector $f_{vv}(u,0)$ points in the tangential direction of
the sectional cusp at $\mb c(u)$.
So we call $\theta(u)$ the {\it cuspidal angle function}.
\item By using $\theta(u)$,
the {\it singular curvature} $\kappa_s$ and 
the {\it limiting normal curvature} $\kappa_\nu$
along the edge of $f$ (cf. \cite{SUY})
are given in \eqref{eq:sn2}.
\end{itemize}
The following fact is important:

\begin{lemma}[\cite{SUY}]
The singular curvature is intrinsic. 
In particular, it is defined along the singular curve
with respect to a given Kossowski metric {\rm (cf. \cite[(2.17)]{HHNSUY}).}
More precisely,
\begin{equation}\label{eq:ks0}
\kappa_s(u)=\frac{-E_{vv}(u,0)}2
\end{equation}
holds, where $(u,v)$ is the coordinate system as in Lemma \ref{fact:coord}.
\end{lemma}

\begin{proof}
As shown in \cite[Proposition 1.8]{SUY}, $\kappa_s$
is expressed as
\begin{equation}\label{eq:ksN}
\kappa_s=\frac{-F_vE_u+2EF_{uv}-EE_{vv}}{2E^{3/2}\lambda_v},
\end{equation}
where $(u,v)$ is a local coordinate system such that
the $u$-axis is the singular set and $\partial_v$
points in the null direction.
If $(u,v)$ is the local coordinate system as in Lemma \ref{fact:coord},
then $F=0$, $\lambda=v \sqrt{EG_0}$
and $E(u,0)=1$ hold. So we can obtain \eqref{eq:ks0}.
\end{proof}

We now prove the following theorem under the assumption that
the curve $\mb c$ is real analytic:

\begin{theorem}\label{thm:S1}
We let $U$ be an open subset of the $uv$-plane $\R^2$ 
containing $J\times \{0\}$ and
$ds^2$ a  
real analytic
Kossowski metric 
satisfying \eqref{eq:3-5}. 
Suppose that 
the curvature function $\kappa$
of the curve $\mb c$ 
is positive everywhere
and the absolute value of the singular curvature 
$\kappa_s(u)$ of $ds^2$ along the singular curve
$$
J\ni u\mapsto (u,0)\in U
$$
is less than $\kappa(u)$ for each $u\in J$.
Then there exist two real analytic generalized cuspidal edges
$g_+,\,\,g_-$ defined on an open subset $V(\subset U)$
containing $J\times \{0\}$ satisfying the following properties:
\begin{enumerate}
\item 
The maps $u\mapsto g_+(u,0)$ and  $u\mapsto g_-(u,0)$
parametrize $C$, which induce
the same orientation as $\mb c:J\to \R^3$.
\item $ds^2$ is the common first fundamental form of $g_+$ and $g_-$.
\item $g_-$ is a faithful isomer of $g_+$.
\item If $\kappa^\pm_\nu:J\to \R$ are the limiting normal curvature functions of  
$g_\pm$, then
$\kappa^-_\nu=-\kappa^+_\nu$ holds on $J$. 
\item If $ds^2$ is non-parabolic at $(u,0)$, then 
$g_+$ and $g_-$ have cuspidal edges at $(u,0)$.
\end{enumerate}
Moreover, suppose that $h:U\to \R^3$ is a
generalized cuspidal edge whose first fundamental form is
$ds^2$. 
If $u\mapsto h(u,0)$
parametrizes $C$ 
giving the same orientation as $\mb c:J\to \R^3$,
then $h$ coincides with $g_+$ or $g_-$. 
\end{theorem}

We prove this theorem from here on out, as a modification of the proof 
given in \cite{NUY}.

\begin{remark}\label{rmk:V}
For each $t_0\in J$, we can 
take a connected local coordinate neighborhood $(V(t_0); u,v)$
of $(t_0,0)$ satisfying  (1), (2) and (3)
of Lemma \ref{fact:coord}.
Since $J$ is compact, we can find finite points
$t_1,...,t_k\in J$ such that $\{V(t_j)\}_{j=1}^k$ covers
the singular curve $J\times \{0\}$.
It is sufficient to prove Theorem \ref{thm:S1} by replacing
$U$ by each $V(t_j)$ ($j=1,...,k$). (In fact, 
the assertion of Theorem \ref{thm:S1}
contains the uniqueness of $g_{\pm}$ on each $V(t_j)$,
and so $g_\pm$ obtained in $V(t_j)$ can be uniquely extended to
$V(t_j)\cup V(t_{j+1})$ for each $j=1,...,k-1$.)
\end{remark}

The statements of Theorem \ref{thm:S1} are properties
of the maps $g_\pm$ which do not depend on the choice of a
local coordinate system containing $J\times \{0\}$.
As explained in Remark \ref{rmk:V},  
we may assume the existence of a local coordinate system $(U;u,v)$
satisfying (1), (2) and (3) of Lemma \ref{fact:coord}, 
without loss of generality.
Then $U$ contains a bounded closed interval $I$
on the $u$-axis such that 
$I\times \{0\}$ gives the singular set of $ds^2$.
We now show the existence of
a real analytic generalized cuspidal edge 
$g(u,v)$ such that
$g(u,0)=\mb c(u)$, $g_v(u,0)=\mb 0$ and
$$
g_u\cdot g_u=E,\quad g_u\cdot g_v=0,\quad
g_v\cdot g_v=G,
$$
which is defined on a neighborhood of $I\times \{0\}$ in $U$
using the Cauchy-Kowalevski theorem.  
$($We remark that
$\mb c(u)$ is parametrized as an
arc-length parameter.$)$
As in Lemma \ref{fact:coord},
we can write $G=v^2 G_0/2$.
The following lemma holds:

\begin{lemma}\label{lem:lemma}
If there exists a real analytic generalized cuspidal edge $g\,(=g_\pm)$
as in Theorem~\ref{thm:S1}, then it is a solution of
the following system of partial differential equations
\begin{equation}\label{eq:CK}
\begin{cases}
g_v&=v \zeta,\\
\xi_v&(=g_{uv})=v \zeta_u, \\
\zeta_v&=\dy\frac1{4}\left((\zeta,g_u,\xi_u)^{T}\right)^{-1}
\biggl((G_0)_v,
-v(G_0)_u,
2r-v (G_0)_{uu}
+4v \zeta_u\cdot \zeta_u\biggr)^{T}
\end{cases}
\end{equation}
of unknown $\R^3$-valued functions $g,\xi,\zeta$
with the initial data
\begin{equation}\label{eq:CKI}
g(u,0)=\mb c(u), \quad \xi(u,0)=\mb c'(u)(=g_u(u,0)), \quad \zeta(u,0)=\mb x(u),
\end{equation}
on  $I$, where $A^{T}$ denotes the transpose of a $3\times 3$-matrix $A$
and
\begin{equation}\label{eq:gvv2}
\mb x(u):=\cos \theta(u) \mb n(u)\mp \sin \theta(u) \mb b(u),
\quad
\cos \theta(u):=\frac{\kappa_s(u)}{\kappa(u)}.
\end{equation}
\end{lemma}

\begin{remark}
Since $g_v=v\zeta$ and $\xi_v=v\zeta_u$, we have
$
\xi_v=v\zeta_u=g_{uv}.
$
Thus, the initial condition $\xi(u,0)=g_u(u,0)$ yields 
$\xi(u,v)=g_u(u,v)$.
\end{remark}

\begin{proof}[Proof of Lemma \ref{lem:lemma}]
Since $ds^2$ is real analytic,
$E$ and $G$ are real analytic functions.
Since $g_v(u,0)=\mb 0$,
we can write
$$
g_v(u,v)=v \zeta(u,v),
$$
where $\zeta(u,v)$ is a real analytic function defined 
on a neighborhood of $I\times \{0\}$ in $\R^2$.
Then 
\begin{equation}\label{eq:I0}
\zeta_v \cdot \zeta=\frac{(\zeta\cdot \zeta)_v}2=\frac{(G_0)_v}{4}.
\end{equation}
 On the other hand, since
\begin{equation}\label{eq:Ia}
   v g_u\cdot \zeta=g_u\cdot g_v=0,
\end{equation}
we have $g_u\cdot \zeta=0$.
Differentiating this, we have
$$
0=v(\zeta\cdot g_u)_v
=v\zeta_v\cdot g_u+v\zeta\cdot g_{uv}
=v\zeta_v\cdot g_u+g_{v}\cdot g_{uv}
=v\zeta_v\cdot g_u+\frac{G_u}2.
$$
Since $G=v^2G_0/2$, we have
\begin{equation}\label{eq:Ib}
\zeta_v\cdot g_u=-\frac{v}{4}(G_0)_u.
\end{equation}
We now obtain information on $\zeta_v\cdot g_{uu}$.
It holds that
$$
  v\zeta\cdot g_{uu}=g_v\cdot g_{uu}
   =(g_v\cdot g_{u})_u-g_{uv}\cdot g_u
   =-g_{uv}\cdot g_u
   =-\frac{E_v}{2},
$$
that is, we obtain
\begin{equation}\label{eq:Ic0}
\zeta\cdot g_{uu}=-\frac{E_v}{2v}.
\end{equation}
On the other hand, we have that 
\begin{align*}
\zeta \cdot g_{uu}+v \zeta_v\cdot g_{uu}
&=g_{vv} \cdot g_{uu}  
=(g_{vv} \cdot g_{u})_u-g_{vvu} \cdot g_{u}  \\
&=\left\{(g_{v} \cdot g_{u})_v-
(g_{v} \cdot g_{uv})\right\}_u-(g_{uv} \cdot g_{u})_v+g_{uv}\cdot g_{uv}  \\
&=(-G_u/2)_u-(E_v/2)_v+g_{uv}\cdot g_{uv}.
\end{align*}
This, together with \eqref{eq:Ic0}, 
gives the following identity
\begin{equation}\label{eq:Ic2}
\zeta_v\cdot g_{uu}
=\frac{E_v-vE_{vv}}{2v^2}-v\frac{(G_0)_{uu}}{4}+v\zeta_u\cdot \zeta_u.
\end{equation}
Since $E_v(u,0)=0$, the function $E_v/v$ is a real analytic function, and
the function 
\begin{equation}\label{eq:r}
r(u,v):=\frac{E_v-vE_{vv}}{v^2}=\left(\frac{-E_v}{v}\right)_v
\end{equation}
is also real analytic.
By \eqref{eq:Ia}, \eqref{eq:Ib} and \eqref{eq:Ic2},
we have the third equality of
\eqref{eq:CK}
under the assumption that
the $3\times 3$ matrix 
$$
M(u,v):=(\zeta,g_u,\xi_{u})
$$
is  regular,  
where $\xi:=g_u$.
The map $g$ must have
the initial data \eqref{eq:CKI},
where
$$
\mb x(u)=\zeta(u,0)
=\lim_{v\to 0}\frac{g_v(u,v)}{v}=g_{vv}(u,0).
$$
By \eqref{eq:fvv}, $\mb x(u)$ can be written in the
form
\begin{equation}\label{eq:gvv}
(\mb x_+(u):=)\mb x(u)=\cos \theta(u) \mb n(u)-\sin \theta(u) \mb b(u),
\end{equation}
where $\theta(u)$ is the function defined by
\eqref{eq:gvv2} and
$\kappa(u)$ (resp.~$\kappa_s(u)$)
is the curvature function of $\mb c(u)$
(resp. the singular curvature function defined by 
\eqref{eq:ks0}).
In fact, since 
the singular curvature $\kappa_s$ of $ds^2$ is less than 
$\kappa$ on $I$,
there exists a real analytic angular function $\theta:I\to \R$ 
satisfying \eqref{eq:gvv2} and
$$
0<|\theta(u)|<\frac\pi2 \qquad (u\in I). 
$$
Moreover, such a  $\theta$ is determined up to a $\pm$-ambiguity.
In particular,
\begin{equation}\label{eq:gvv3}
(\mb x_-(u):=)\mb x(u)=\cos \theta(u) \mb n(u)+\sin \theta(u) \mb b(u)
\end{equation}
is the other possibility.
\end{proof}

We now return to the proof of Theorem \ref{thm:S1}.
We have
\begin{align*}
(M(u,0)&=)\,\left(\zeta(u,0),\,g_u(u,0),\,g_{uu}(u,0)\right) \\
&\,=\left(\cos \theta(u) \mb n(u)-\sin \theta(u) \mb b(u),
\,\mb e(u),\,\kappa(u) \mb n(u)\right).
\end{align*}
Since the singular curvature of $ds^2$ 
satisfies $|\kappa_s|<\kappa$ on $I$,
the function $\sin \theta$
does not vanish on $I$. 
Thus the matrix $M(u,0)$ is regular for each $u\in I$.
We can then apply the Cauchy-Kowalevski theorem (cf. \cite{KP})
for
the system of partial differential equations \eqref{eq:CK}
with initial data \eqref{eq:CKI}
and obtain a unique real analytic solution $(g,\xi,\zeta)$
of \eqref{eq:CK}
defined on a neighborhood of $I\times \{0\}$ in $\R^2$.
Thus, we obtained the existence of real analytic generalized 
cuspidal edges $g_{\pm}(u,v)$ 
corresponding to the initial data $\mb x_\pm(u)$.
By the above construction of these $g_{\pm}$,
the functions $\pm \theta$ coincide with
the cuspidal angles of $g_\pm$, respectively.
To accomplish the proof of Theorem \ref{thm:S1}, 
we need to verify that
the first fundamental forms of $g_\pm$ coincide with $ds^2$.
To show this, we consider the case $g=g_+$
with initial condition $\mb x(u):=\mb x_+(u)$,
without loss of generality.
The third equation of \eqref{eq:CK}
yields
$
  \zeta_v\cdot \zeta
  = (G_0)_v/4,
$
and hence we have
$( \zeta\cdot \zeta - G_0/2 )_v =0$.
Since
$$
  \zeta(u,0)\cdot \zeta(u,0) - \frac{G_0(u,0)}{2} 
  = \mb x(u)\cdot \mb x(u) -1 
  =0,
$$
the Cauchy-Kowalevski theorem yields that
\begin{equation}\label{eq:G2}
\zeta\cdot \zeta=\frac{G_0}2.
\end{equation}
Hence, 
by the first equation of \eqref{eq:CK},
we have 
\begin{equation}\label{eq:GG}
  g_v \cdot  g_v
  =\frac{v^2 G_0}2
  =G.
\end{equation}
On the other hand, using \eqref{eq:CK}, we have
$$
(\xi-g_u)_v=\xi_v-g_{uv}=v\zeta_u-(g_{v})_u
=v\zeta_u-(v\zeta)_u=0.
$$
The initial condition $\xi(u,0)=g_u(u,0)$ yields that
$
g_u=\xi.
$
Then $g_{uv}=\xi_v=v\zeta_u$
and 
$$
g_{uv}\cdot \zeta=v \zeta_u\cdot \zeta
=v\frac{(\zeta\cdot \zeta)_u}2=\frac{v (G_0)_u} 4
$$
hold. Using this, we have
$$
(g_u \cdot \zeta)_v
=g_{uv}\cdot \zeta+g_u\cdot \zeta_v \\
=\frac{v (G_0)_u} 4-\frac{v (G_0)_u} 4=0.
$$
Since
$g_u(u,0) \cdot \zeta(u,0)=0$,
we can conclude that
$g_u \cdot \zeta=0$,
that is,
\begin{equation}\label{eq:G}
g_{u}\cdot g_v=0
\end{equation}
is obtained.  We now prepare the following:

\begin{lemma}
Suppose that $($which is one of the conditions in 
\eqref{eq:CK}$)$
$$
\zeta_v \cdot \xi_u(=\zeta_v\cdot g_{uu})=
\frac{2r-v (G_0)_{uu}
+4v \zeta_u\cdot \zeta_u}{4}.
$$
Then the initial condition \eqref{eq:gvv}
implies the following identity
\begin{equation}\label{eq:key}
\frac{E_v}2+v \zeta \cdot \xi_{u}=0.
\end{equation}
\end{lemma}

\begin{proof}
Using \eqref{eq:G2},
we have that
\begin{align*}
(\zeta \cdot \xi_{u})_v
&=\zeta_v\cdot \xi_{u}+\zeta \cdot \xi_{uv} 
=\zeta_v\cdot \xi_{u}+\zeta \cdot g_{uuv} 
=\zeta_v\cdot \xi_{u}+\zeta \cdot (v \zeta_{uu}) \\
&=\frac1{4}\biggl(2r-v (G_0)_{uu}+4v \zeta_u\cdot \zeta_u\biggr)+\zeta 
\cdot (v \zeta_{uu})\\
&=\frac{r}2-\frac{v}2 (G_0)_{uu}+v(\zeta_u\cdot \zeta_u+\zeta \cdot \zeta_{uu}) \\
&=\frac{r}2-\frac{v}{4} (\zeta\cdot \zeta)_{uu}+\frac{v}2(\zeta\cdot \zeta)_{uu} 
=\frac{r}2.
\end{align*}
By \eqref{eq:r},  
$$
\left(\zeta \cdot \xi_{u}+\frac{E_v}{2v}\right)_{\!v}=0
$$
holds. On the other hand, we have
\begin{align*}
\zeta(u,0) \cdot \xi_{u}(u,0)
&={\mb x}(u)\cdot g_{uu}(u,0)
=(\cos \theta(u) \mb n(u)-\sin \theta(u) \mb b(u))\cdot \mb c''(u) \\
&=\bigl(\cos \theta(u) \mb n(u)-\sin \theta(u) \mb b(u)\bigr)
\cdot \left(\kappa(u)\mb n(u)\right)
=\kappa(u)\cos \theta(u) \\
&=\kappa(u) \frac{\kappa_s(u)}{\kappa(u)}=\kappa_s(u)
=\frac{-E_{vv}(u,0)}2=\lim_{v\to 0}\frac{-E_v(u,v)}{2v}.
\end{align*}
So we obtain \eqref{eq:key}.
\end{proof}

We again return to the proof of Theorem \ref{thm:S1}.
By \eqref{eq:key},
we have
\begin{align*}
\frac12 (g_u\cdot g_u)_v
=g_{uv}\cdot g_u
=(g_v\cdot g_u)_u-g_{v}\cdot g_{uu}
 =-g_v \cdot g_{uu}
=\frac{E_v}2. 
\end{align*}
This, with the initial condition
$
g_{u}(u,0)\cdot  g_u(u,0)
=\mb c'(u)\cdot \mb c'(u)=1
$
implies
\begin{equation}\label{eq:E}
g_{u}\cdot g_u=E.
\end{equation}
By \eqref{eq:E},  \eqref{eq:G} and \eqref{eq:GG},
we can conclude that $ds^2$ coincides with
the first fundamental form of $g=g_+$, which implies
the existence and uniqueness of $g=g_+$.
Replacing $\theta$ by $-\theta$,
we also obtain the existence and uniqueness
of $g=g_-$.
Since the cuspidal angles of $g_\pm$ are distinct,
the image of $g_-$ does not coincide with $g_+$.
Since the orientation of 
$u\mapsto g_-(u,0)$ is
compatible with that of the
curve $u\mapsto g_+(u,0)$, the map
$g_-$ is a faithful isomer of $g_+$.

Here, we suppose $ds^2$ is non-parabolic at $(u,0)$, 
then $g_+$ and $g_-$ are wave fronts by \cite[Proposition 4 (o)]{HNUY}. 
Since $ds^2$ is of type ${\rm I}$, the criterion of cuspidal edges
given in \cite[Proposition 4 (i)]{HNUY} yields that
$g_+$ and $g_-$ are both cuspidal edges.

Finally, the last assertion
of Theorem \ref{thm:S1}
follows from the uniqueness of
the system of partial equations
\eqref{eq:CK}
as a consequence of  the Cauchy-Kowalevski theorem,
proving Theorem \ref{thm:S1}.

\medskip
By the above proof of Theorem \ref{thm:S1},
we obtain the following:

\begin{corollary}\label{cor:T}
The cuspidal angle of $g_-$ is $-\theta$, 
where $\theta$ is the cuspidal angle of $g_+$.
In particular, $g_-$ is a faithful isomer of $g_+$
since $\sin \theta\ne 0$.
\end{corollary}

We next prove the following:

\begin{lemma}\label{lemma:S1}
Let $U$ be an open subset of the $uv$-plane $\R^2$ 
containing $J\times \{0\}$, and let
$ds^2$ be a real analytic
Kossowski metric of type I
defined on $U$
satisfying $(1)$--$(3)$ of Lemma \ref{fact:coord}.
Suppose that  
the singular set of $ds^2$ consists only of non-parabolic points. 
If there exist  open subsets $V_i(\subset U)$ $(i=1,2)$
containing $J\times \{0\}$ 
and a diffeomorphism $\phi:V_1\to V_2$
such that $\phi^*ds^2=ds^2$ 
and $\phi(u,0)=(u,0)$ hold for $u\in J$,
then $V_1=V_2$ and $\phi$ is the identity map.
\end{lemma}

\begin{proof}
Let $\mb c(u)$ ($u\in J$) be a space curve satisfying the assumption of
Theorem \ref{thm:S1}, and
let $g_+$ be one of cuspidal edges realizing $ds^2$ as in Theorem \ref{thm:S1}.
Since $g_+\circ \phi$ and $g_+$
have the common first fundamental form $ds^2$,
the last assertion of Theorem \ref{thm:S1} yields that
$g_+\circ \phi$ coincides with either $g_+$
or $g_-$.
Since $g_+\circ \phi$ and $g_+$ have the same image,
they have a common cuspidal angle at each point
of $C$. 
So there exists a symmetry $T$ of $C$ such that $T\circ g_+\circ \phi=g_+$. 
Suppose $T$ is not the identity map. Since $\phi(u,0)=(u,0)$, $\phi$ maps 
the domain $D_+:=\{v>0\}$ to $D_-:=\{v<0\}$. However, it is impossible, 
because $\phi^*ds^2=ds^2$ and the Gaussian curvature on $D_+$ takes 
the opposite sign of that on $D_-$ (cf. [5, (1.14)]). 
Thus, $T$ is the identity map and $g_+\circ \phi=g_+$ holds. 
Since the singular set of $g_+$ consists of cuspidal edge points,
 $g_+$ is injective, and $\phi$ must be the identity map.
\end{proof}

\begin{proposition}\label{prop:ds2}
Let $ds^2$ be a real analytic Kossowski metric belonging to
$\mathcal K_*^\omega(\R^2_o)$. Suppose that $\phi$
is a local $C^\omega$-diffeomorphism 
satisfying $\phi^*ds^2=ds^2$ and $\phi(o)=o$
which is not the identity map.
Then $\phi$ is an involution which
reverses the orientation of the singular curve.
Moreover, such a $\phi$ is uniquely determined.
\end{proposition}

\begin{proof}
We can take a local coordinate system
satisfying $(1)$--$(3)$ of Lemma \ref{fact:coord}.
Since $\phi(o)=o$, the fact that $u\mapsto (u,0)$ is
the arc-length parametrization with respect to $ds^2$
yields that
either $\phi(u,0)=(u,0)$ or $\phi(u,0)=(-u,0)$ holds.
If $\phi(u,0)=(u,0)$, then
by Lemma \ref{lemma:S1}, $\phi$ is the
identity map, a contradiction.
So we have $\phi(u,0)=(-u,0)$.
This means that $\phi$
reverses the orientation of the singular curve. 
In this situation, we have
$\phi\circ \phi(u,0)=(u,0)$.
Applying Lemma \ref{lemma:S1} again,
$\phi\circ \phi$ is the identity map,
that is, $\phi$ is an involution.
We next suppose that $\psi$ is 
another
local $C^\omega$-diffeomorphism 
satisfying $\psi^*ds^2=ds^2$ and $\psi(o)=o$.
Then $\phi\circ \psi(u)=(u,0)$ holds,
and Lemma \ref{lemma:S1} yields that
$\phi\circ \psi$ is the identity map.
So $\psi$ must coincide with $\phi$.
\end{proof}
	
\begin{corollary}\label{cor:ds2}
Let $ds^2_f$ be a real analytic Kossowski metric 
as the first fundamental form of
$f\in \mc G^\omega_{*,3/2}(\R^2_J,\R^3,C)$.
Suppose that $\phi$ is a $C^\omega$-symmetry of $ds^2_f$,
then it is effective and
is an involution reversing the
orientation of the singular curve.
\end{corollary}

\begin{proof}
Without loss of generality, we may assume that
the parameters $(u,v)$ of $f(u,v)$ 
satisfy (1)-(3) of Lemma \ref{fact:coord}
for  $ds^2_f$.
Let $P$ be the midpoint of $C$ with respect to the
arc-length parameter. Then there exists $c\in J$
such that $f(c,0)=P$.
Thinking $o:=(c,0)$, we may regard
$f$ belongs to $\mc G^\omega_{*,3/2}(\R^2_o,\R^3,C)$.
Since $f\in \mc G^\omega_{*,3/2}(\R^2_J,\R^3,C)$,
by restricting $f$ to a neighborhood of
$o$, the metric $ds^2_f$ can be considered as
an element of $\mathcal K_*^\omega(\R^2_o)$ (cf. [5, (2) of Theorem A]).
So the symmetry $\phi$ of $ds^2_f$ 
satisfies the desired property by
Proposition \ref{prop:ds2}.
Since $\phi$ is real analytic, the property
is extended on a tubular neighborhood of
the singular curve.
\end{proof}

Moreover, the following important
property for symmetries of Kossowski metrics
is obtained:

\begin{theorem}
Let $p$ be a singular point of a real analytic Kossowski metric $ds^2$
which is an accumulation point of non-parabolic 
singular points of type I.
Suppose that  $\phi$ is 
a  local $C^\omega$-diffeomorphism
fixing $p$ satisfying $\phi^*ds^2=ds^2$.
Then $\phi$ is  an involution and
reverses the orientation of the singular curve
if it is not the identity map.
\end{theorem}

\begin{proof}
Let $\gamma(t)$ be a real analytic parametrization
of the singular curve of the real analytic
Kossowski metric $ds^2$ such that $\gamma(0)=p$.
We let $\{p_n\}_{n=1}^\infty$ be a sequence of non-parabolic points
converging to $p$.
Since $\gamma$ is real analytic, the existence of such a sequence implies
that, for sufficiently small $\epsilon(>0)$,
$\gamma((-\epsilon,0)\cup (0,\epsilon))$ consists of
non-parabolic points of type {\rm I}.
Then 
$$
s(t):=\int_0^t\sqrt{ds^2(\gamma'(u),\gamma'(u))}\,du
\qquad (t\in (-\epsilon,\epsilon))
$$
is a monotone increasing function of $t$, giving
a continuous parametrization of $\gamma$.
Using this parameter $s$, either 
$\phi\circ \gamma(s)=\gamma(s)$ or $\phi\circ \gamma(s)=\gamma(-s)$ holds.
If the former case happens, then 
applying Proposition \ref{prop:ds2} at a 
non-parabolic point $\gamma(s)$ ($s\ne 0$), 
$\phi$ must be the identity map on a neighborhood of
$\gamma(s)$. Since $\phi$ is
real analytic, it must be the identity map 
on a neighborhood of $p$.

We next consider the case
that $\phi\circ\gamma(s)=\gamma(-s)$.
Then $\phi\circ \phi\circ \gamma(s)=\gamma(s)$, and
the above argument implies that $\phi$ is an
involution, proving the assertion.
\end{proof}

\begin{proof}[Proof of Theorem {\rm I}]
Let $ds^2_f$ be the first fundamental form of $f$.
Then $ds^2_f$ is a Kossowski metric of type {\rm I}, by
Proposition \ref{lem:GK}.
Since $f$ belongs to $\mc G_*^\omega(\R^2_J,\R^3,C)$ (cf. \eqref{eq:kks}), 
the singular curvature 
$\kappa_s$ of $ds^2_f$ is less than $\kappa$ on $J$.
By Theorem \ref{thm:S1},
there exist two generalized cuspidal edges
$g_+,g_-\in \mc G_*^\omega(\R^2_J,\R^3,C)$ whose
first fundamental forms coincide with $ds^2_f$.
Since $ds^2_f$ is the first fundamental form of $f$,
the last assertion of Theorem 3.8 yields that
either $f=g_+$ or $f=g_-$ holds.
Without loss of generality, we may set $f=g_+$,
then 
$\check f:=g_-$
is the desired isometric dual of $f$.
The remaining assertions for $f\in \mc G^\omega_*(\R^2_o,\R^3,C)$
follow from 
Lemma~\ref{lem:Red}.
\end{proof}

\begin{defi}\label{def:check}
For 
each $f\in {\mc G}^\omega_*(\R^2_o,\R^3,C)$
(resp. $f\in {\mc G}^\omega_*(\R^2_J,\R^3,C)$),
we call the above $\check f\in {\mc G}^\omega_*(\R^2_o,\R^3,C)$ 
(resp. $\check f\in {\mc G}^\omega_*(\R^2_J,\R^3,C)$)
the {\it isometric dual} of
$f$.
\end{defi}

\section{A representation formula for generalized cuspidal edges}\label{sec PII}

We set $J=[-l,l]$ ($l>0$).
Let $\mb c:J\to \R^3$ be an embedding
with arc-length parameter
whose curvature function $\kappa(u)$ is positive
everywhere.
We denote by $\mb e(u):=\mb c'(u)$, 
and by $C$ the image of $\mb c$.
We let $\mb n(u)$ and $\mb b(u)$ be
the unit principal normal vector field and
unit binormal vector field of 
$\mb c(u)$, respectively. 
We fix a sufficiently small $\delta(>0)$ and
consider a map given by
\begin{equation}\label{eq:repF}
f(u,v):=\mb c(u)+
(A(u,v),B(u,v))\pmt{
\cos \theta(u) & -\sin \theta(u) \\
\sin \theta(u) & \cos \theta(u) 
}\pmt{\mb n(u) \\ \mb b(u)},
\end{equation}
where $u\in J$ and $|v|<\delta$.
Here $A(u,v)$, $B(u,v)$ and $\theta(u)$ 
are $C^r$-functions, and
satisfy
$$
A(u,0)=A_v(u,0)=0,\quad A_{vv}(u,0)\ne 0,
\quad
B(u,0)=B_v(u,0)=B_{vv}(u,0)=0.
$$
Then it can be easily checked that
any generalized cuspidal edges along $C$ are right
equivalent to one of such a map.
Moreover, if
$B_{vvv}(u,0)\ne 0$,
then $f$ is a cuspidal edge along $C$.
The function $\theta(u)$ is called the 
{\it cuspidal angle} at $\mb c(u)$.
Let $\kappa(u)$ be the curvature of $\mb c(u)$.
Then the $C^r$-functions defined by
\begin{equation}\label{eq:kskn}
\kappa_s(u)=\kappa(u)\cos\theta(u), \qquad
\kappa_\nu(t)=\kappa(u)\sin\theta(u)
\end{equation}
give the {\it singular curvature} 
and the {\it limiting normal curvature} 
respectively.
The map germ $f$ can be determined by
$$
(\theta(u),\,A(u,v),\,B(u,v)).
$$
We call these functions {\it Fukui's data}.

\begin{defi}
In the expression \eqref{eq:repF}, if
\begin{itemize}
\item $u$ is an arc-length parameter of $\mb c$,
\item for each $u\in J$, the map
$
(-\delta, \delta)\ni t\mapsto (A(u,t),B(u,t))\in \R^2
$
is a generalized cusp at $t=0$ (called a {\it sectional cusp at $u$}), and
$t$ gives a normalized half-arc-length parameter
(see the appendix),
\end{itemize}
then the expression \eqref{eq:repF} of $f$ 
by setting $v=t$ as the normalized half-arc-length parameter
is called the {\it normal form} of a generalized cuspidal edge.
\end{defi}

We now fix such a normal form $f$. We set
\begin{equation}\label{eq:a23}
\pmt{\mb v_2(u) \\ \mb v_3(u)}=
\pmt{
\cos \theta(u) & -\sin \theta(u) \\
\sin \theta(u) & \cos \theta(u) 
}\pmt{\mb n(u) \\ \mb b(u)},
\end{equation}
then we have
\begin{equation}\label{eq:a23i}
f(u,t)=\mb c(u)+A(u,t)\mb v_2(u)+B(u,t) \mb v_3(u).
\end{equation}

\begin{defi}\label{def:Fdata}
Let $(a,b)$ ($a<b$) be an interval on $\R$, and
$\delta\in (0,\infty]$ a positive number.
A $C^r$-differentiable ($r=\infty$ or $r=\omega$)
quadruple $(\kappa,\tau,\theta, \hat\mu)$
is called a {\it fundamental data} (or a {\it modified Fukui-data})
if \begin{itemize}
\item $\kappa:(a,b)\to \R$ is a $C^r$-function
such that $\kappa>0$,
\item $\tau,\theta:(a,b)\to \R$ and $\hat\mu:(a,b)\times (-\delta,\delta)\to \R$ are 
$C^r$-functions.
\end{itemize}
\end{defi}

Summarizing the above discussions,
one can easily show the following representation formula
for generalized cuspidal edges,
which is a mixture of Fukui's representation formula
as in \cite[(1.1)]{F} for generalized cuspidal edges and 
a representation formula for cusps in the appendix
 (cf. Lemma \ref{lem:A1}):

\begin{proposition}
\label{prop:Bj}
Let $(\kappa,\tau,\theta, \hat\mu)$ be a given fundamental data
and $\mb c(u)$ $(u\in J)$ the space curve with arc-length 
parameter whose curvature function and 
torsion function are $\kappa(u)$ and $\tau(u)$.
Then, 
\begin{equation}\label{eq:repF2}
f(u,t):=\mb c(u)+
(A(u,t),B(u,t))\pmt{
\cos \theta(u) & -\sin \theta(u) \\
\sin \theta(u) & \cos \theta(u) 
}\pmt{\mb n(u) \\ \mb b(u)}
\end{equation}
gives a generalized cuspidal edge written in a normal form
along $C:=\mb c(J)$,
where $(A,B)$ is given by
\begin{equation}\label{eq:ab}
(A(u,t),B(u,t))=\int_0^t v 
(\cos \lambda(u,v),\sin \lambda(u,v))
dv,
\,\,
\lambda(u,t):=\int_0^t \hat\mu(u,v)dv.
\end{equation}
Moreover, 
\begin{enumerate}
\item $\theta$ gives the cuspidal angle of $f$ along $\mb c$,
\item $t\mapsto \hat\mu(u,t)$ is the function
given in \eqref{A2} for the sectional cusp of $f$ at $u$.
\end{enumerate}
Furthermore, any generalized cuspidal 
edge along $C$ is right equivalent to such an $f$ 
constructed in this manner $($see also Remark \ref{Z}$)$.
\end{proposition}

\begin{remark}\label{rem:key}
Let $\mb c_0(u)$ be a space curve
parametrized by the arc-length parameter $u$ 
defined on an interval $J:=[-l,l]$ ($l>0$),
 whose curvature function and torsion function are
$\kappa(u)$ and $\tau(u)$, respectively.
We assume that $\mb c_0(0)=\mb 0$.
Suppose that $C:=\mb c_0(J)$ admits a non-trivial symmetry $T$.
Since $\mb 0$ is the midpoint of $C$ and is fixed by $T$,
we may assume that $T\in \op{O}(3)$
and set 
$
\sigma:=\det(T)\in \{1,-1\}.
$
Then $\mb c_1(u):=T \mb c_0(-u)$
is a space curve whose curvature function and torsion function are
$\kappa(u)$ and $\sigma \tau(u)$ respectively.
We denote by
$\mb e_i(u)(:=\mb c'_i(u)),\, \mb n_i(u)$ and $\mb b_i(u)$
$(i=0,1)$
the unit tangent vector, unit principal normal vector
and unit binormal vector of $\mb c_i(u)$,
respectively.
Differentiating $T\circ \mb c_0(u)=\mb c_1(u)$,
we have
\begin{align*}
&T\mb e_{0}(-u)=T\circ \mb c'_0(-u)=-\mb c'_1(u)=
-\mb e_{1}(u), \\
& \kappa_0(-u) T\mb n_0(-u)=T\circ \mb c''_0(-u)
=\mb c''_1(u)=\kappa_1(u)\mb n_1(u).
\end{align*}
In particular,  
$
T\mb e_0(-u)=-\mb e_1(u),\,\,T\mb n_0(-u)=\mb n_1(u)
$
and $\kappa_0(-u)=\kappa_1(u)$ hold,
where $\kappa_i$ ($i=1,2$) is the curvature function of $\mb c_i$.
Since $\sigma:=\det(T)\in \{1,-1\}$, we have
$$
\mb b_0=\mb e_0 \times \mb n_0=
(-T\mb e_1)\times (T\mb n_1)
=-T\left(\mb e_1\times \mb n_1\right)=
-\sigma T\mb b_1.
$$
Using this, one can also obtain the relation 
$-\sigma\tau_0(-u)=\tau_1(u)$,
where $\tau_i$ ($i=1,2$) is the torsion function of $\mb c_i$.
We set
$$
f_i:=\mb c_i+
(A_i,B_i)\pmt{
\cos \theta_i & -\sin \theta_i \\
\sin \theta_i & \cos \theta_i 
}\pmt{\mb n_i \\ \mb b_i}
\qquad (i=0,1),
$$
and suppose
$$
A_0(-u,t)=A_1(u,t),\quad
B_0(-u,t)=-\sigma B_1(u,t),\quad
\theta_0(-u)=-\sigma\theta_1(u).
$$
Then
\begin{align*}
&T\circ f_0(-u,t)\\
\phantom{aaa}&=T\mb c_0(-u)+
(A_0(-u,t),B_0(-u,t))
\pmt{\cos \theta_0(-u) & -\sin \theta_0(-u) \\
       \sin \theta_0(-u) & \cos \theta_0(-u) }
\pmt{T\mb n_0(-u) \\ T\mb b_0(-u)} \\
&=
\mb c_1(u)+(A_1(u,t),-\sigma B_1(u,t))
\pmt{\cos (-\sigma\theta_1(u)) & -\sin (-\sigma\theta_1(u)) \\
       \sin (-\sigma\theta_1(u)) & \cos (-\sigma\theta_1(u))}
\pmt{\mb n_1(u) \\ -\sigma \mb b_1(u)} \\
&=
\mb c_1(u)+
(A_1(u,t),B_1(u,t))
\pmt{\cos \theta_1(u) & -\sin \theta_1(u) \\
       \sin \theta_1(u) & \cos \theta_1(u) }
\pmt{\mb n_1(u) \\ \mb b_1(u)} 
=f_1(u,t).
\end{align*}
Thus, we obtain the relation
$
f_1(u,t)=T\circ f_0(-u,t).
$ 
In particular, $f_1$ has the same first fundamental
form as $f_0$.  Moreover,
\begin{enumerate}
\item[(a)] 
if $T\in \op{SO}(3)$, then the 
cuspidal angle of $f_1$ takes opposite sign of that of $f_0$.
By the uniqueness of the isometric dual of $f_0$
(cf. Theorem \ref{thm:S1}),  $\check f_0(u,t)=f_1(u,t)=T\circ f_0(-u,t)$
holds, that is, $f_1$ is the faithful isomer (i.e. the
isometric dual) of $f_0$.
\item[(b)] 
if $T\in \op{O}(3)\setminus \op{SO}(3)$, then the 
cuspidal angle of $f_1$ coincides with that of $f_0$.
Then 
$f_0(u,t)=f_1(u,t)=T\circ f_0(-u,t)$ holds 
(cf. Theorem \ref{thm:S1}), that is,
the image of $f_0$ is invariant by $T$.
\end{enumerate}
\end{remark}

\begin{remark}\label{rem:key2}
Let $f(u,t)$ be a generalized cuspidal edge associated to 
the data $(\kappa(u),\tau(u),\theta(u), \hat \mu(u,t))$.
Then $f_\#(u,t):=f(-u,t)$ is also
a generalized cuspidal edge along the same space curve 
as $f$ but with the reversed orientation.
If we set $\mb c_\#(u):=\mb c(-u)$, 
then $\mb c_\#(u)=f_\#(u,0)$ holds.
By a similar calculation like as in Remark \ref{rem:key},
one can easily verify that
$(\kappa(-u),-\tau(-u),-\theta(-u), \hat \mu(-u,t))$
gives the fundamental data of $f_\#(u,t)$.
\end{remark}

We next prove Theorem {\rm II} in the introduction.

\begin{proof}[Proof of Theorem {\rm II}]
We fix $f\in \mc G^\omega_{**}(\R^2_J,\R^3,C)$ arbitrarily.
We denote by $ds^2_f$ the first fundamental form of $f$.
Since $f$ is admissible,
the singular curvature $\kappa_s(u)$ satisfies
\eqref{eq:kks3}, and so 
\eqref{eq:kks2} holds.
By Theorem \ref{thm:S1},
there exist two distinct generalized cuspidal edges $g_\pm$ 
whose first fundamental forms coincide with $ds^2_f$
such that $g_+=f$, and 
$u \mapsto g_{-}(u,0)$ 
has the same orientation as
that of  $u \mapsto f(u,0)$. 
Since $f$ is admissible,
the singular curvature $\kappa_s$ is determined only by $ds^2_f$.
Thus $g_\pm$ belong to $\mc G_{**}^\omega(\R^2_J,\R^3,C)$.
By the proof of Theorem I, we know that
$\check f:=g_{-}$ gives the isometric dual of $f$. 

On the other hand, we replace $u$ with $-u$
(that is, the orientation of $C$ is reversed).
Since $f$ is admissible, it holds
that
$$
0<|\kappa_s(u)|\le \min_{u\in J}\kappa(u)<\kappa(-u)\qquad (u\in J).
$$
So, applying Theorem~\ref{thm:S1} again,
there exist two distinct generalized cuspidal edges 
$h_\pm\in \mc G_{**}^\omega(\R^2_J,\R^3,C)$
such that
$u \mapsto h_{\pm}(u,0)$ 
have the same orientation as
that of $u \mapsto f(-u,0)$. 
Then $ds^2_f$ gives the  common first fundamental form
of the generalized cuspidal edges $h_{\pm}$. 
By \eqref{eq:gvv2}, we may assume that
the cuspidal angle $\theta_*(u)$ (resp. $-\theta_*(u)$) 
($\theta_*(u)\theta(u)>0$)
of $h_+$ (resp. $h_-$) satisfies 
$$
\cos \theta_* (u)=\frac{\kappa_s(u)}{\kappa(-u)}.
$$
Since the orientation of the singular curves of $h_\pm$ is 
opposite of that of $f$, the two maps
$h_\pm$ are non-faithful isomers of $f$.
We set 
$$
\text{ $f_*:=h_+$ (the inverse), and
$\check f_*:=h_-$ (the inverse dual)}.
$$
By the above Remark \ref{rem:key2},
the cuspidal angle of $f_\#(u,v):=f(-u,v)$ is
$-\theta(-u)$,
the cuspidal angle $\theta_* (u)$
takes opposite sign of that 
of $f_\#(u,v)$. So the image of $f$ does not coincide with
that of $f_*$. Hence $f_*$ is an isomer of $f$.

By our construction of $f_*$, (1), (2) and (3) are obvious.
So we prove (4). We suppose that
the first fundamental form of
a generalized cuspidal edge $k\in \mc G^\omega_{**}(\R^2_I,\R^3,C)$ 
is isometric to $ds^2_f$.
(The case that $k\in \mc G^\omega_{*}(\R^2_o,\R^3,C)$
is obtained by Lemma \ref{lem:Red}.)
Since the first fundamental form is determined independently
of a choice of local coordinate system, 
we have
$
\mc J_C(f\circ\phi)=\mc J_C(f)\circ\phi,
$
where $\phi$ is a diffeomorphism on
a certain tubular neighborhood of $J\times \{0\}$.
So we may assume that $ds^2_k=ds^2_f$ without loss of generality. 
Then $k$ must coincide with one of $\{g_+,g_-,h_+,h_-\}$,
because of the uniqueness of the solution of
\eqref{eq:CK} with initial condition \eqref {eq:CKI}.
\end{proof}

\begin{defi}
We call the above $f_*$ and $\check f_*$ the {\it inverse} 
and the {\it inverse  dual} of $f\in {\mc G}^\omega_{**}(\R^2_J,\R^3,C)$,
respectively.
\end{defi}

We next give criteria of a given germ of generalized cuspidal edge 
to be a cuspidal edge, cuspidal cross cap or $5/2$-cuspidal edge
in terms of the extended half-cuspidal curvature function $\hat \mu$.

\begin{proposition}
\label{prop:3}
Let $f\in \mc G^r(\R^2_J,\R^3,C)$ be the generalized cuspidal edge associated to 
a fundamental data $(\kappa,\tau,\theta, \hat \mu)$. Then
\begin{enumerate}
\item $f$ gives a cuspidal edge along the $u$-axis
if $\hat \mu(u,0)\ne 0$,
\item $f$ gives a cuspidal cross cap at $o$
if $\hat \mu(0,0)=0$ and $\hat \mu_u(0,0)\ne 0$,
\item $f$ gives a $5/2$-cuspidal edge along the $u$-axis
if $\hat \mu(u,0)=0$ and $\hat \mu_{vv}(u,0)\ne 0$.
\end{enumerate}
\end{proposition}

The first and the second assertions have been proved
in \cite[Proposition 1.6]{F}.

\begin{proof}
We may assume that $f$ is written in a normal form.
The first assertion follows from (1) of Proposition A.2.
The second assertion follows from
the criterion for cuspidal cross caps given 
in \cite{FSUY}, but can be proved much easier 
using (2) of \cite[Proposition 4.4]{F}.
The third assertion is a consequence of (2)
of Proposition A.2.
\end{proof}

To compute the first and the second fundamental forms
of $f$ in terms of fundamental data,
the following Frenet-type formula for singular curves is convenient.

\begin{lemma}[Izumiya-Saji-Takeuchi \cite{IST}
and Fukui \cite{F}]\label{lem:FIST}
The following formula holds
$($cf. \eqref{eq:a23}$)$:
\begin{equation}\label{eq:F-Frenet}
\pmt{\mb e' \\ \mb v'_2 \\ \mb v'_3}
=
\pmt{
0 & \kappa \cos \theta & \kappa \sin \theta \\
-\kappa\cos \theta & 0 & \tau-\theta' \\
-\kappa\sin \theta & -(\tau-\theta') & 0 }
\pmt{\mb e \\ {\mb v}_2 \\ {\mb v}_3}.
\end{equation}
\end{lemma}

This formula can be rewritten as (cf. \eqref{eq:a23})
$$
\pmt{\mb e' \\ \mb v'_2 \\ \mb v'_3}
=
\pmt{
0 & \kappa_s & \kappa_\nu \\
-\kappa_s & 0 & \kappa_t \\
-\kappa_\nu & -\kappa_t & 0 }
\pmt{\mb e \\ {\mb v}_2 \\ {\mb v}_3},
$$
which is the one given in Izumiya-Saji-Takeuchi
\cite[Proposition 3.1]{IST},
where  $\kappa_t$ is the  {\it cusp-directional torsion}  
defined in \cite{MS} and has the expression (cf. \cite[Page 7]{F})
\begin{equation}\label{eq:k_t}
\kappa_t=\tau-\theta'.
\end{equation}
Using Lemma \ref{lem:FIST}, one can easily  obtain
the following by a straightforward
computation:

\begin{proposition}
[{Fukui \cite{F}}]\label{prop:EFG}
The first fundamental form 
$
ds^2_f=Edu^2+2Fdudt+Gdt^2
$
of $f$ as in \eqref{eq:repF2} is
given by
\begin{align}\label{eq:49}
E&=(1-(A \cos \theta+B \sin \theta)\kappa)^2+
(A_u+(\theta'-\tau)B)^2+
(B_u-(\theta'-\tau)A)^2, \\ \nonumber
F&=A_t(A_u+(\theta'-\tau)B)
+B_t(B_u-(\theta'-\tau)A),\quad
G=t^2,
\end{align}
where $\kappa,\,\tau,\,\theta$ are functions of $u$
and $A,B$ are functions of $(u,t)$.
\end{proposition}

\begin{proof}
Differentiating $f=\mb c+A \mb v_2+B \mb v_3$,
we have
\begin{align*}
f_u&=(1-(A \cos \theta+B \sin \theta)\kappa)\mb e
+(A_u+(\theta'-\tau)B)\mb v_2+
(B_u-(\theta'-\tau)A)\mb v_3, \\
f_t&=A_t \mb v_2+B_t \mb v_3.
\end{align*}
Since $E=f_u\cdot f_u$,
$F=f_u\cdot f_t$ and $G=f_t\cdot f_t$,
we obtain the assertion.
\end{proof}

\setcounter{equation}{9}
We can write
$$
\hat\mu(u,t)=\mu_0(u)+\mu_1(u)t+\mu_2(u)t^2+\mu_3(u,t)t^3,
$$
and then Lemma A.1 yields that
\begin{align}\label{eq:a1}
A
&=\frac{t^2}{2}-\frac{\mu_0(u)^2}{8}t^4 -\frac{\mu_0(u) \mu_1(u)}{10} t^5
+t^6a_6(t,u), \\
\label{eq:a2}
B &=
\frac{\mu_0(u)}{3}t^3 +\frac{\mu_1(u)}{8}t^4 +\frac{2\left(-\mu_0(u)^3+2\mu_2(u)\right)}{30}t^5 
+t^6b_6(t,u),
\end{align}
where $a_6(t,u)$ and $b_6(t,u)$ denote $C^r$-functions. 

\begin{corollary}\label{cor:expG}
The Gaussian curvature $K$ of $ds^2_f$ 
satisfies
$$
K(u,t)=\frac{K_0(u)}t+K_1(u)+K_2(u)t+K_3(u,t)t^2,
$$
where
\begin{align*}
&K_0:=\mu_0 \kappa_{\nu}, \quad 
K_1:=
-\kappa_s\mu_0^2-\kappa_t^2+\kappa_\nu \mu_1,  \\
&K_2:=
-\frac{\kappa_\nu \mu_0^3}2
+\frac{\kappa_s\kappa_\nu \mu_0}2
-\frac{3\kappa_s \mu_0\mu_1}2
+\kappa_\nu \mu_2
-2\mu'_0\kappa_t
+\frac{\mu_0}2\kappa'_t,
\end{align*}
and $K_3(u,t)$ is a $C^r$-function.
Here $\kappa_s, \kappa_\nu$ and $\kappa_t$ are
defined in \eqref{eq:sn2} and
\eqref{eq:k_t}. Moreover, 
$\mu_0=\kappa_c/2$ $($cf. \eqref{eq:m0}$)$
and $\kappa'_t=d \kappa_t(u)/du$.
\end{corollary}

Fukui \cite[Theorem 1.8]{F} has already determined 
the first two terms $K_0$ and $K_1$.
So the essential part of the above corollary
is the statement for $K_2$.

\begin{proof}

One can obtain this formula by computing the
sectional curvature of $ds^2_f$, or
alternatively, one can get it by 
computing the second fundamental form of $f$
as Fukui did in \cite{F}.
In each approach, 
\eqref{eq:a1} and \eqref{eq:a2} play
crucial roles.
\end{proof}

As a consequence of this corollary, the first term
$$
K_0:=\mu_0 \kappa_\nu=\frac{\kappa_c \kappa_\nu}{2}
$$
defined in \cite{MSUY}
is an intrinsic invariant,
which is called the {\it product curvature}.
The second term 
$K_1$
is an intrinsic invariant. 
We consider the term $K_2$.
Since $K_0=\kappa_c \kappa_\nu/2$, and since
$\mu_0$ is equal to the cuspidal curvature $\kappa_c$, 
the fact that $\kappa_s$ and $\kappa_c\kappa_\nu$ are intrinsic yields that
$$
\tilde K_2:=
-\frac{\kappa_\nu \mu_0^3}2
-\frac{3\kappa_s \mu_0\mu_1}2
+\kappa_\nu \mu_2
-2\mu'_0\kappa_t
+\frac{\mu_0}2\kappa'_t
$$
is also an intrinsic invariant.
Using this, we can prove the following assertion:

\begin{proposition}
\label{prop:4}
Let $f\in \mc G^r(\R^2_J,\R^3,C)$ 
be the generalized cuspidal edge associated to 
a fundamental data $(\kappa,\tau,\theta, \hat\mu)$ satisfying $\sin \theta\ne 0$. 
Then
\begin{enumerate}
\item $f$ gives a cuspidal edge along the
$u$-axis if $K_0(u)\ne 0$,
\item $f$ gives a cuspidal cross cap at $u=0$
if $K_0(0)=0$ and $dK_0(0)/du=0$, and
\item $f$ gives a $5/2$-cuspidal edge along the
$u$-axis if $K_0(u)=0$ and $K_2(u)\ne 0$.
\end{enumerate}
In particular, these conditions depend only on
the first fundamental form of $f$.
\end{proposition}

\begin{proof}
Since $\sin \theta(u)\ne 0$, we have $\kappa_\nu(u)\ne 0$.
Since $K_0=\mu_0\kappa_\nu$,
$K_0(u)=0$ if and only if $\mu_0(u)=0$. 
Since $\mu_0(u)=\hat \mu(u,0)(=\kappa_c(u))$,
the first and second assertions follow from (1) and (2)
of Proposition \ref{prop:3}, respectively.
On the other hand, if $\mu_0(=\kappa_c)$ is identically zero, then
$
K_2=
\kappa_\nu \mu_2.
$
So $K_2(u)\ne 0$ if and only if $\mu_2(u)\ne 0$. 
Thus, the third 
assertion immediately follows from
(3) of Proposition~\ref{prop:3}.
\end{proof}

We now prove Fact \ref{f2} in the introduction.

\begin{proof}[Proof of Fact \ref{f2}]
Since $\sin \theta \ne 0$ if and only if $\kappa_\nu \ne 0$,
the assertions (1) and (2) follow from Theorem \ref{thm:S1}.
We next prove (3).
We remark that
\begin{align*}
\mc K^\omega_*(\R^2_o)&=\{ds^2_f\in
\mc K^\omega_{\rm I}(\R^2_o)\,;\,K_0(0)\ne 0\},\\
\mc K^\omega_{p,*}(\R^2_o)&=\{ds^2_f\in 
\mc K^\omega_{\rm I}(\R^2_o)\,;\,K_0(0)= 0,\,\,dK_0(0)/du\ne 0\},\\
\quad \mc K^\omega_{a,*}(\R^2_o)&=\{ds^2_f\in 
\mc K^\omega_{\rm I}(\R^2_o)\,;\,K_0(u)=0,\,\,K_2(0)\ne 0\}
\end{align*}
hold in terms of our coordinates $(u,t)$.
We have shown the following (cf. Propositions \ref{prop:3} and \ref{prop:4}).
\begin{itemize}
\item $K_0(0)\ne 0$  if and only if $\mu_0(0)(=\kappa_c(0))\ne 0$.
\item
$K_0(0)= 0$ and $dK_0(0)/du\ne 0$
if and only if $\mu_0(0)(=\kappa_c(0))=0$
and $d\mu_0(0)/du\ne 0$. 
\item $K_0(u)= 0$ and $K_2(0)\ne 0$ if and only if
$\mu_0(u)=0$ and $\mu_2(0)\ne 0$.
\end{itemize}
By Corollary \ref{cor:crr}, 
the following assertions hold: 
\begin{itemize}
\item $\hat K(o)\ne 0$ if and 
only if $K_0(0)\ne 0$.
\item $\hat K(o)=0$ and $\partial\hat K(o)/\partial u\ne 0$ 
if and only if $K_0(0)= 0$ and $dK_0(0)/du\ne 0$.
\end{itemize}
So the first fundamental form
$ds^2_f$ of $f$  
belongs to 
$\mc K^\omega_{*}(\R^2_o)$
(resp.  $\mc K^\omega_{p,*}(\R^2_o)$)
if and only if 
$\mu_0(0)(=\kappa_c(0))\ne 0$
(resp. $\mu_0(0)(=\kappa_c(0))=0$
and $d\mu_0(0)/du\ne 0$). 
On the other hand,
$ds^2_f$ belongs to 
$\mc K^\omega_{a,*}(\R^2_o)$
if and only if
$\mu_0(u)=0$ and $\mu_1(0)\ne 0$.
In fact, $\eta:=\partial/\partial t$ 
gives the null direction of $f$ along the $u$-axis
(as the singular curve of $ds^2_f$), and we have (cf. \eqref{eq:Kv})
$
dK(\eta)=K_t(u,0)=K_2(u).
$
\end{proof}

\rm
Finally, we consider the cuspidal edges
with vanishing limiting normal curvature:
A cuspidal edge is called {\it asymptotic} 
if its first fundamental form is asymptotic (see Section 1), 
which is equivalent to the condition that 
the cuspidal angle $\theta(u)$ of $f$
is constantly equal to $0$ or $\pi$ along its edge.

If $f$ is an asymptotic cuspidal edge,
the singular curvature $\kappa_s$, limiting normal curvature 
$\kappa_\nu$   and
cusp-directional torsion $\kappa_t$ 
satisfy
\begin{equation}
\kappa_s=\epsilon \kappa,\quad \kappa_\nu=0, \quad
\kappa_t=\tau,
\end{equation}
where $\epsilon:=\cos \theta\, (\in \{1,-1\})$. 
So we get the following:

\begin{proposition}
\label{prop:4b}
Let $f\in \mc G^r_{3/2}(\R^2_J,\R^3,C)$ be a cuspidal edge associated to 
a fundamental data $(\kappa,\tau,\theta, \hat\mu)$. 
If $\sin \theta$ vanishes identically, then 
\begin{enumerate}
\item the limiting normal curvature $\kappa_\nu$ vanishes identically,
\item the first fundamental form of $f$ is an 
asymptotic Kossowski metric, and
\item the Gaussian curvature $K$ of $f$ can be extended
across its singular set as a $C^r$-function.
\end{enumerate}
Moreover, the  sign of $K$ coincides with
the sign of $
(K_1=)-\epsilon \kappa\mu_0^2-\tau^2
$
whenever $K_1\ne 0$, where $\epsilon:=\cos \theta$.
\end{proposition}

As an application, we first consider the case $K$ vanishes identically.

\begin{corollary}\label{prop:Zero}
Let $f\in \mc G^r_{3/2}(\R^2_J,\R^3,C)$ be the cuspidal edge 
whose Gaussian curvature $K$ vanishes identically. 
Then $C$ is a regular space curve whose 
torsion function does not vanish,
and $f$ is the tangential developable of $C$.
In particular, $f$ has no isomers.
\end{corollary}

\begin{proof}
Since $K$ vanishes identically, the identity
$
-\epsilon \kappa\mu_0^2=\tau^2
$
holds along $C$. Since $f$ is a cuspidal edge,
$\mu_0$ has no zeros, and the left hand side
does not vanish. Thus, the torsion function $\tau$
of $C$ also has no zeros.
Since $f$ is a wave front, its principal directions
along $C$ are well-defined (cf. \cite[Proposition 1.6]{MU}).
Moreover, each singular point of $f$ is disjoint from
umbilical set (cf. \cite[Proposition 1,10]{MU}), 
and the zero principal curvature 
direction is uniquely determined at each point of $C$.
Moreover, it can be easily seen that 
this direction must be the tangential direction of $C$.
Since $K$ vanishes identically,
$f$ must be a ruled surface (cf. \cite[Proposition 2.2]{MU}), 
so it must be the tangential developable of $C$.
\end{proof}

\begin{remark}
The standard cuspidal edge $f_0(t)=(u^2,u^3,v)$ 
does not satisfy the assumption of
Corollary \ref{prop:Zero}, since
the singular set image is a line.
\end{remark}

We next consider the case $K>0$.
If $\theta=\pi$ 
and $\mu_0$ is sufficiently large,
then the Gaussian curvature $K$ near the singular set 
can be positive.  
So we can construct cuspidal edges with $K>0$.
The following assertion 
is an immediate consequence of
Proposition~\ref{prop:4b}.

\begin{corollary}\label{prop:ka}
Let $f\in \mc G^r_{3/2}(\R^2_J,\R^3,C)$ 
be the cuspidal edge  
whose Gaussian curvature $K$ is 
bounded near singular set and positive, then 
it is asymptotic satisfying $\theta=\pi$ and $\kappa_s<0$.
\end{corollary}

The negativity of $\kappa_s$
has been pointed out in \cite{SUY}.
Although Theorem \ref{thm:S1} does not cover 
the case $\kappa_\nu=0$,
Brander \cite{B} showed the existence of cuspidal edges
in the case of $K=1$ along a given space curve 
$C$ of $\kappa_\nu>0$ using the loop group theory.

\section{Relationships among isomers}\label{sec:6}

In this section, we show several  properties of isomers, and  prove
the last two statements in the introduction.
We fix a space curve $\mb c(u)$ 
satisfying ${\mathbf c}(0)={\mathbf 0}$ which is parametrized
by arc-length defined on a closed interval
$J:=[-l,l]$ ($l>0$) whose curvature function
$\kappa(u)$ is positive everywhere.
We prove the following:

\begin{proposition}\label{cor:key0}
Let $f\in {\mc G}^\omega_{*,3/2}(\R^2_J,\R^3,C)$.
Then 
$\check f$ is congruent
$($cf. Definition \ref{def:distinct}$)$
to 
$f$ if and only if
\begin{enumerate}
\item $C$ lies in a plane, or 
\item $C$ has a positive non-trivial symmetry
and the first fundamental form
$ds^2_f$ has an effective symmetry
$($cf. Definition \ref{def:ds-sym2}$)$. 
\end{enumerate}
\end{proposition}

\begin{proof}
We suppose that $\check f$ is congruent to $f$.
By Remark \ref{rem:key2},
it is sufficient to consider the case that $C$
does not lie in any plane.
By Remark \ref{Z},
there exist an isometry $T$ on $\R^3$ and 
a diffeomorphism $\phi$ defined
on a neighborhood of the singular curve of $f$
such that
\begin{equation}\label{eq:ID}
T\circ f\circ \phi=\check f.
\end{equation}
We consider the case that $T$ fixes each point of $C$.
Then $C$ must lie in a plane, a contradiction.
So $T$ is a non-trivial 
symmetry of $C$, that is, it reverses the orientation of $C$.
We suppose that $T$ is a negative symmetry.
Then 
(b) of Remark \ref{rem:key} implies that
the image of $f$ coincides with that of $T\circ f$.
Since the image of $\check f$ is different from that of $f$,
this case never happens.
So $T$ must be a positive symmetry, and then $\phi$ gives
an effective symmetry of $ds^2_f$.

Conversely, 
if $C$ has a positive non-trivial symmetry
and the first fundamental form
$ds^2_f$ has 
an effective symmetry 
$\phi$, then $T\circ f\circ \phi$
is a faithful isomer of $f$ 
as seen in (a) of Remark \ref{rem:key}. 
Since such an isomer is uniquely determined (cf. Theorem~\ref{thm:S1}),
we have \eqref{eq:ID}.
\end{proof}

\begin{remark}\label{rem:C}
Suppose that $C$ is planar and $S$ is the reflection 
with respect to the plane containing $C$.
For each $f\in {\mc G}^r_{*,3/2}(\R^2_J,\R^3,C)$,
$S\circ f$ gives
a faithful isomer of $f$.
Moreover, if $f$ is real analytic (i.e. $r=\omega$),
then we have $\check f=S\circ f$
(cf. Definition~\ref{def:check}).
\end{remark}
 
\begin{example}\label{thm5}
Let $f\in \mc G^\infty_*(\R^2_J,\R^3,C)$ 
be an admissible generalized cuspidal edge
whose fundamental data is $(\kappa,\tau,\theta,\hat\mu)$
($\tau\ne 0$).
Suppose that $\kappa,\tau$ and $\theta$ are constant,
and the extended half-cuspidal curvature function $\hat\mu$ does not depend
on $u$. In this case, without assuming the real analyticity of $f$,
we can show the existence of an isometry 
$T\in \op{SO}(3)$ and an effective symmetry
$\phi$ of $ds^2_f$ such that $T\circ f \circ \phi$ gives 
a faithful isomer of $f$ as follows:
In fact, in this case $C$ has the constant curvature
$\kappa$ and the constant torsion $\tau$.
Since $\tau\ne 0$, 
$C$ is a helix in $\R^3$ and there exists
a $180^\circ$-rotation $T\in \op{SO}(3)$ with respect to the
principal normal line at $\mb 0\in C$ such that $T(C)=C$.
By the first part of
Proposition \ref{prop:Sym2},
it is sufficient to show that
the first fundamental form 
$$
ds^2_f=E(t)du^2+2F(t)du dt+G(t) dt^2
$$
of $f$ admits an effective symmetry $\phi$
as an involution.
In fact, if such a $\phi$
exists, then $(\check f:=)T\circ f\circ \phi$ gives
the isometric dual of $f$.
In this situation, two functions $A,B$
can be expressed as (cf. \eqref{eq:49} and \eqref{eq:ab})
$A(t):=t^2 \alpha(t)$ and $B(t):=t^3 \beta(t)$,
where $\alpha(t)$ and $\beta(t)$ are $C^r$-functions.
By Proposition~\ref{prop:EFG},
\begin{itemize}
\item $E(t)$ is positive for each $t$,
\item there exists a $C^\infty$-function $F_0(t)$ such that
$
F(t)=t^4 F_0(t) 
$,
and
$G(t)=t^2$.
\end{itemize}
Setting
$$
\omega_1=\sqrt{E(t)}\left(du+\frac{F(t)}{E(t)}dt\right),\qquad
\omega_2=t \sqrt{\frac{E(t)-t^6F_0(t)^2}{E(t)}}dt,
$$
we have
$
ds^2_f=(\omega_1)^2+(\omega_2)^2.
$
Moreover, if we set
\begin{equation}
x(u,t):=u+\int_0^t\frac{F(v)}{E(v)}dv,
\qquad y(t):=\int_0^t\sqrt{\frac{E(v)-v^6F_0(v)^2}{E(v)}}dv.
\end{equation}
Then we can take $(x,y)$ as a new local coordinate system
centered at $(0,0)$, 
and $t$ can be considered as a function of $y$.
So we can write $t=t(y)$,
and 
$$
ds^2_f=E(y)dx^2+t(y)^2 dy^2.
$$
So the local diffeomorphism
$
\phi:(x,y) \mapsto (-x,y)
$
gives an 
effective symmetry
of $ds^2_f$.

Regarding the fact that the
fundamental data of $f$ is $(\kappa,\tau,\theta,\mu)$,
we show in later that $\check f$ is right equivalent to 
the cuspidal edge whose fundamental data of $(\kappa,\tau,-\theta,\mu)$,
see Proposition \ref{thm:V0}.
\end{example}

\begin{proof}[Proof of Theorem {\rm III}]
Suppose that $ds^2_f$ admits a symmetry $\phi$.
Then this symmetry is effective
(cf. Corollary \ref{cor:ds2}). So, 
$f\circ \phi$ and $\check f\circ \phi$ must be right
equivalent to $\check f_*$ and $f_*$, respectively.
In particular, the number of right equivalence 
classes of $f,\check f,f_*,\check f_*$
is two.
 
Conversely, we suppose that two of $\{f,\check f,f_*,\check f_*\}$ are
right equivalent.
Replacing $f$ by $\check f,\,\, f_*,\,\, \check f_*$,
we may assume that one of the right equivalent pair is
$f$ and the other is $g\in \{\check f,\,\, f_*,\,\, \check f_*\}$.
Without loss of generality, we may assume that
$f$ is written in a normal form.
Since $\check f$ cannot be right equivalent to $f$, the map
$g$ must be right equivalent to $f_*$ or $\check f_*$,
that is, there exists a local diffeomorphism $\phi$ such that
$g=f\circ \phi$, which implies
$\phi^*ds^2_f=ds^2_f$.
If $\phi$ is an identity map, then $g=f$ holds.
However, it contradicts the fact that
$u\mapsto f(u,0)$ and
$u\mapsto f_*(u,0)=\check f_*(u,0)$ 
give mutually distinct 
orientations to $C$.
So, by Corollary \ref{cor:ds2}, $\phi$ must be an effective symmetry of $ds^2_f$.
\end{proof}

\begin{corollary}\label{1}
Let $f\in {\mc G}^{\omega}_{**,3/2}(\R^2_J,\R^3,C)$.
Suppose that
\begin{itemize}
\item[(1)] $C$ is planar and does not admit 
any non-trivial symmetry at $\mb 0$, and
\item[(2)] $ds^2_f$ admits no effective symmetries
$($cf. Definition \ref{def:ds-sym2}$)$.
\end{itemize}
Then 
\begin{itemize}
\item $\check f:=S\circ f$ holds,
where $S\in \op{O}(3)$ is the reflection with respect to the plane
containing $C$,
\item 
the isometric dual, inverse and the inverse dual are
given by $S\circ f$, $f_*$ and $S\circ f_*$, respectively.
Moreover, $f_*$
is not congruent to $f$.
\end{itemize}
In particular, the four maps consist of two congruence classes.
\end{corollary}

\begin{proof}
As seen in Remark \ref{rem:C},
$\check f:=S\circ f$ holds.
We next prove the second assertion.
Since $C$ lies in a plane, $\mc I_C(f)=S\circ f$ holds.
By applying Theorem {\rm II},
the right equivalence classes of $\mc J^{-1}_{C}(\mc J_C(f))$ are
represented by
$\{f, \,\,S\circ f,\,\,f_*, \,\,S\circ f_*\}$.
It is sufficient to show that $f_*$ is not congruent
to $f$.  If not, then, by Remark \ref{Z},
there exist $T\in \op{O}(3)$ and 
a diffeomorphism $\phi$ defined
on a neighborhood of the singular curve of $f$
such that
$
T\circ f_*\circ \phi=f
$.
In particular, $\phi^*ds^2_f=ds^2_f$ holds.
By (1), $T$ is not non-trivial.
So, $\phi$ must be an effective symmetry,
contradicting (2).
\end{proof}

We next consider the case that $ds^2_f$ has an 
effective symmetry.

\begin{proposition}\label{2}
Let $f\in {\mc G}^\omega_{**,3/2}(\R^2_J,\R^3,C)$.
Suppose that
\begin{itemize}
\item[(1)] $C$ is non-planar and does not admit 
any non-trivial 
symmetry at $\mb 0$, 
\item[(2)] $ds^2_f$ admits an 
effective symmetry $\phi$.
\end{itemize}
Then 
$\check f(:=\mc I_C(f))$ is not congruent to  $f$, 
and
$\check f,\,\, \check f\circ \phi$ and $f\circ \phi$
give the isometric dual, inverse and inverse dual, respectively.
\end{proposition}

\begin{proof}
By Proposition \ref{cor:key0},
$\check f$ is not congruent to $f$.
Since $\check f\circ \phi$ 
(resp. $f\circ \phi$) has
the same first fundamental form as $f$,
the fact that $\phi$ is effective yields
that it coincides with either $f_*$ or $\check f_*$.
Since the cuspidal angle of $\check f\circ \phi$ 
(resp. $f\circ \phi$) takes 
the opposite sign 
(resp. the same sign) of that of $f$ (cf. Remark \ref{rem:key2}),
we have $f_*=\check f\circ \phi$ (resp. $\check f_*=f\circ \phi$).
\end{proof}

\begin{corollary}\label{3}
Let $f\in {\mc G}^{\omega}_{**,3/2}(\R^2_J,\R^3,C)$.
Suppose that
\begin{itemize}
\item[(1)] $C$ is planar and does not admit  
any non-trivial symmetry at the origin $\mb 0$, 
\item[(2)] $ds^2_f$ admits an 
effective symmetry $\phi$.
\end{itemize}
Then 
\begin{itemize}
\item $\check f=S\circ f$ holds,
where $S\in \op{O}(3)$ is the reflection with respect to the plane
containing $C$.
\item 
Moreover, 
$S\circ f,\,\, S\circ f\circ\phi, f\circ \phi$ 
give the isometric dual, inverse and inverse dual, respectively.
\end{itemize}
As a consequence, all of isomers are congruent to $f$.
\end{corollary}

\begin{proof}
As we have seen in 
Remark \ref{rem:C},
$\check f=S\circ f$ holds.
Since $S\circ f\circ \phi$ 
(resp. $f\circ \phi$) has
the same first fundamental form as $f$,
the fact that $\phi$ is effective yields
it coincides with $f_*$ or $\check f_*$.
Since the sign of cuspidal angle of $S\circ f\circ \phi$
(resp. $f\circ \phi$)
along the curve $\mb c_\#(u):=\mb c(-u)$
takes the opposite sign (resp. the same sign)
of that of $f$, we have
$f_*=S\circ f\circ \phi$ (resp. $\check f_*=f\circ \phi$).
Finally, it is obvious that the four maps
are congruent. So the proposition is proved.
\end{proof}

We then consider the case that $C$ has a 
non-trivial symmetry.

\begin{proposition}\label{4}
Let $f\in {\mc G}^\omega_{**,3/2}(\R^2_J,\R^3,C)$.
Suppose that 
\begin{itemize}
\item[(1)] $C$ is non-planar and admits a non-trivial
symmetry $T\in \op{O}(3)$ at $\mb 0$,
\item[(2)] $ds^2_f$ does not admit 
any effective symmetries.
\end{itemize}
Then 
\begin{itemize}
\item $\check f:=\mc I_C(f)$ is not congruent to  
$f$, and
\item 
$T\circ \check f,\,\,T\circ f$
are the inverse and inverse dual, respectively.
\end{itemize}
In particular, 
$f, \,\, \check f,\,\, T\circ \check f$ and $T\circ f$ 
consist of two congruence classes.
\end{proposition}

\begin{proof}
By Proposition \ref{cor:key0},
$\check f$ is not congruent to $f$.
So the assertion  can be shown easily.
\end{proof}

We get the following corollary.

\begin{corollary}\label{cor:symP}
Let $f\in {\mc G}^{\omega}_{**,3/2}(\R^2_J,\R^3,C)$.
Suppose that $C$ lies in a plane and admits  
a non-trivial symmetry $T$ at the origin $\mb 0$. 
Then $\check f=S\circ f$ holds, and
$T\circ f,\,S\circ T\circ f$ give the inverse and the
inverse dual of $f$, 
where $S$ is a reflection with respect to the plane.
As a consequence,
$f,\check f, f_*,\check f_*$ 
belong to a single 
congruence class.
\end{corollary}

\begin{proof}
Obviously, $\check f=S\circ f$ holds
(cf. Remark \ref{rem:C}).
On the other hand, $T\circ f$ gives
a non-faithful isomer,
and its isometric dual $S\circ T\circ f$
also gives another non-faithful isomer.
\end{proof}

\begin{figure}[thb]
\begin{center}
\includegraphics[height=3.5cm]{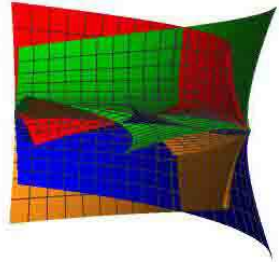}
\caption{
The four cuspidal edges given in Example \ref{ex:symP}
}
\end{center}\label{fig:SymP0}
\end{figure}

\begin{example}\label{ex:symP}
We set
$$
f(u,v):=\biggl(\phi(u,v) \cos u-1,\,\,
\phi(u,v) \sin u,\,\,v^3u+2v^3-v^2\biggr),
$$
where $\phi(u,v):=-v^3 u-2v^3-v^2+1$.
Then, it has cuspidal edge singularities along
$$
\mb c(u)\,(:=f(u,0))=
(\cos u-1,\sin u, 0).
$$
By setting, 
$$
S:=\left(
\begin{array}{ccc}
 1 & 0 & 0 \\
 0 & 1 & 0 \\
 0 & 0 & -1 
\end{array}
\right), \qquad
T:=\left(
\begin{array}{ccc}
 1 & 0 & 0 \\
 0 & -1 & 0 \\
 0 & 0 & 1 \\
\end{array}
\right),
$$
$S\circ f$ is the faithful isomer,
and $T\circ f,\,\, TS\circ f$
are non-faithful isomers. 
We remark that $f$ is associated to 
Fukui's data $(\theta,A,B)$
given by
$$
\theta=\frac{\pi}4, \quad
A(u,v):=\sqrt{2}v^2,\quad  B(u,v):=\sqrt{2}v^3(u+2).
$$
\end{example}

Finally, we consider the case that $C$ 
and $ds^2_f$ admit a symmetry and 
an effective symmetry, respectively.

\begin{proposition}\label{prop:Sym2}
Let $f\in {\mc G}^{\omega}_{**,3/2}(\R^2_J,\R^3,C)$.
Suppose that 
\begin{itemize}
\item[(1)] $C$ is non-planar and admits a non-trivial
symmetry $T\in \op{O}(3)$ at $\mb 0$, 
\item[(2)] $ds^2_f$ admits an 
effective symmetry $\phi$.
\end{itemize}
Then any isomer of $f$ is right equivalent to
one of $\check f,\,\, \check f\circ \phi$,\,\, $f\circ \phi$. 
Moreover,
\begin{itemize}
\item if $T$ is positive $($i.e. $T\in \op{SO}(3))$, then 
$\check f=T\circ f\circ \phi$, and
\item if $T$ is negative $($i.e. $T\not\in \op{SO}(3))$, then 
$\check f$ is not congruent to $f$.
\end{itemize}
\end{proposition}

\begin{proof}
We set $g:=T\circ f\circ \phi$.
If $T$ is positive, then
$g$ is a faithful isomer of $f$ 
as shown in Remark~\ref{rem:key}.
On the other hand,
if $T$ is negative, then
$\check f$ is not congruent to $f$ 
by Proposition \ref{cor:key0} and so it 
not congruent to $f$.
\end{proof}

\begin{proof}[Proof of Theorem {\rm IV}]
We suppose that
$C$ has no non-trivial symmetries, and also $ds^2_f$ has no symmetries.
If two of $\{f, \check f, f_*,\check f_*\}$ are
mutually congruent, replacing $f$ by one of its isomers,
we may assume that $f$ is congruent to $g$, where
$g$ is one of $\{\check f, f_*,\check f_*\}$.
By Proposition~\ref{cor:key0}, we may assume that
$g=f_*$ or $g=\check f_*$.
 Suppose that $g$ is congruent to $f$. 
Then (cf. Remark \ref{Z})
there exist a non-trivial symmetry $T\in \op{O}(3)$ of $C$ and 
a local diffeomorphism $\phi$ 
such that
$$
T\circ g\circ \phi=f.
$$
Since $C$ has no non-trivial symmetries, and $ds^2_f$ has also no symmetries,
$\phi$ is the identity map and $T$ is not a non-trivial
symmetry. However, this contradicts the fact that
$u\mapsto f(u,0)$ and
$u\mapsto f_*(u,0)=\check f_*(u,0)$ 
give mutually distinct 
orientations to $C$.
So we obtained (1).

The assertion (2) follows from
Corollaries \ref{1}, \ref{3},  \ref{cor:symP}
and Propositions \ref{2},
 \ref{4}, and \ref{prop:Sym2}, by using the fact that
any symmetries of $ds^2_f$ are effective
(cf. Corollary \ref{cor:ds2}).

Finally, suppose that $N_f=1$.
We first consider the case that $C$ lies in a plane.
If $C$ has no non-trivial symmetries
and $ds^2_f$ has also no symmetries,
then $N_f=2$ holds by
Corollary \ref{1}.
So either $C$ or $ds^2_f$ has a symmetry.
If $C$ has a symmetry, then
$N_f=1$ by Corollary \ref{cor:symP} (this corresponds to the case (a)).
On the other hand, if $C$ has no non-trivial symmetries and
$ds^2_f$ also has a symmetry $\phi$, then $\phi$ is effective
(cf. Corollary \ref{cor:ds2}). So, 
Corollary \ref{3} yields that $N_f=1$.
(This corresponds to the case (b). In fact, we denote by
$T_0$ the reflection with respect to the plane containing $C$.
We let $T_1$ be a non-trivial symmetry of $C$.
If $T_1$ is positive, then (b) holds obviously.
On the other hand, if $T_1$ is negative, then $T_0\circ T_1$
is a positive symmetry and (b) holds.)

So we may assume that $C$ does not lie in
any planes. 
The assumption $N_f=1$ implies $\check f$ must congruent to
$f$. By Proposition \ref{cor:key0}, this holds
only when (c) happens, since $C$ 
does not lie in
any planes. 
\end{proof}

\section{Examples}

One method to give a numerical approximation of a
isometric dual $g$ of a real analytic cuspidal edge $f$
is to determine the Taylor expansion of $g(u,v)$ 
at $v=0$ along the
$u$-axis as a singular set so that $g=\mc I_C(f)$.
In \cite[Page 85]{NUY}, we give a numerical approximation of the
isometric dual of
$$
f_{0}(u,v)=\left(u,-\frac{v^2}{2}+\frac{u^3}{6},
\frac{u^2}{2}+\frac{u^3}{6}+\frac{v^3}{6}\right).
$$
We denote by $C$ the image of singular curve $u\mapsto f_0(u,0)$.
In the figure of the isometric dual $g_0=\mc I_C(f_0)$ given in \cite[Figure 2]{NUY},  
the surface $g_0$ seems like it is lying on the almost opposite side of $f_0$.
This is the reason why the cuspidal angle 
$\theta(u)$ of $f_0(u,v)$ is $\pi/2$ at $u=0$.
The red lines of Figure~3 (left) indicates the section 
of $f_0,g_0$ at $u=-1/4$.
The orange (resp. blue) surface corresponds to $f_0$ (resp. $g_0$). 
We can recognize that the cuspidal angle takes value less than $\pi/2$,
that is, the normal direction of $g_0$
is linearly independent of that of $f_0$ at $(u,v)=(-1/4,0)$.
On the other hand, Figure~3 (right) indicates the images of
the numerical approximations of 
the two non-faithful isomers $f_1,g_1$ of $f_0$.

\begin{figure}[h!]
\begin{center}
\includegraphics[height=3.2cm]{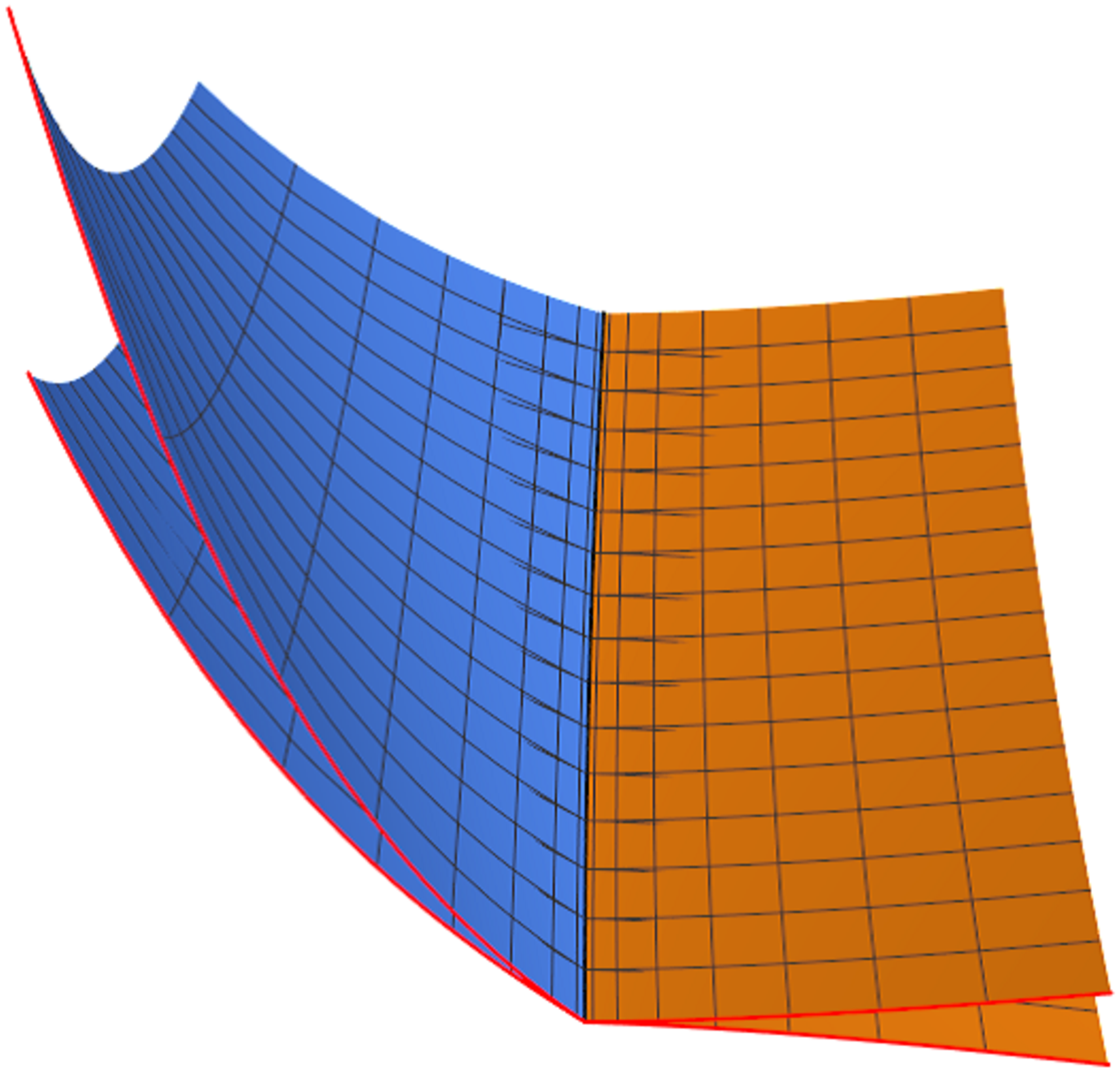}\qquad \quad
\raisebox{-1cm}{\includegraphics[height=3.6cm]{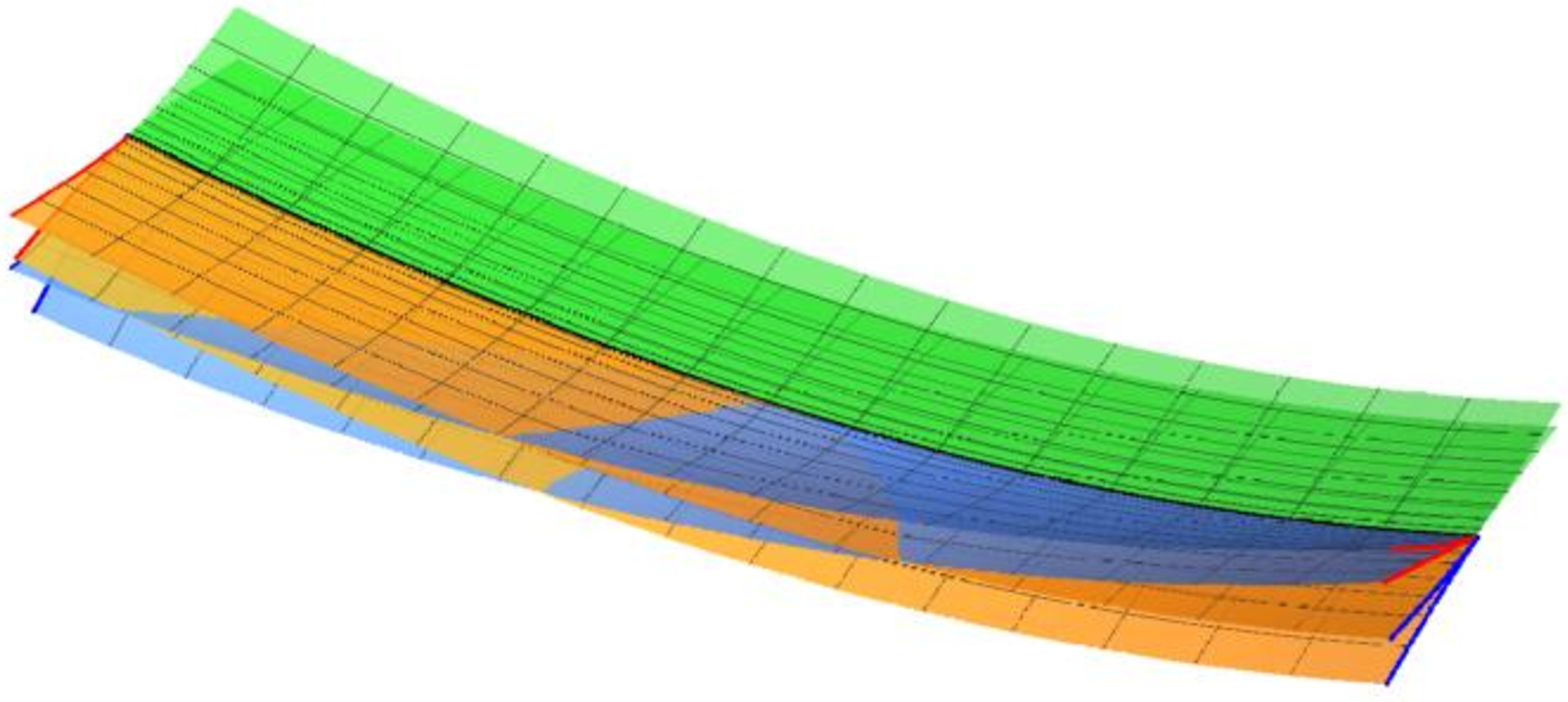}}
\caption{The images of $f_{0},g_0$ (left), and
the images of $f_{0},f_1,g_1$ (right),
where $f_0$ is indicated as the orange surfaces.}
\end{center}\label{fig:NUY}
\end{figure}

By Proposition \ref{prop:EFG},
one can easily observe that 
the first fundamental form of
$f_{-\theta}$ does not coincide with 
that of $f_\theta$.
This means that the image of $f_{-\theta}$ cannot
coincide with that of $f_{\theta}$ nor
$\check f_{\theta}$.
However, 
one might expect the
possibility
that $f_{-\theta}$ is 
an isomer of $f_\theta$.
Here, we consider the case that 
the space curve $C$ has a non-trivial symmetry $T$.
In this case, we know that
$
f,\,\, \check f,\,\,  T\circ f,\,\,  T\circ \check f
$
are only the possibilities of isomers.
Thus, if $f_{-\theta}$ is 
an isomer of $f_\theta$, then
it must be congruent to either $f$ or $\check f$.
We give here the following two propositions
which are related to one of these possibilities
(by the following 
Proposition \ref{thm:V0}, Example \ref{thm5}
is just the case that $f_{-\theta}$ is right equivalent
to $\check f$.) 

\begin{proposition}\label{thm:V0}
Let $C$ be a space curve which 
admits a non-trivial symmetry $T\in \op{SO}(3)$ at $\mb 0$,
and let $f:=f_\theta\in \mc G^\infty(\R^2_J,\R^3,C)$
be a generalized cuspidal edge as in
the formula 
\eqref{eq:repF}
such that
\begin{itemize}
\item $T\circ f(-u,0)=f(u,0)$, and
\item the cuspidal angle $\theta$
satisfies $\theta(u)=\sigma\theta(-u)$ 
where $\sigma\in \{+,-\}$.
\end{itemize}
Suppose that $A(u,v)$ and $B(u,v)$ satisfy 
one of the following two conditions:
\begin{enumerate}
\item $A(-u,-v)=A(u,v)$ and
$B(-u,-v)=-B(u,v)$ or
\item $A(-u,v)=A(u,v)$ and $B(-u,v)=-B(u,v)$.
\end{enumerate}
Then 
$f_\theta=T\circ f_{-\sigma \theta}\circ \phi$
holds, where $\phi(u,v)=(-u,-v)$ $($resp. $\phi(u,v)=(-u,v))$
in the case of {\rm (1)} $($resp. {\rm (2)}$)$. 
In particular, $f_{-\theta}$ is a 
right equivalent to $\check f$ 
if $\sigma=+$, and
the image of $f$ is invariant under $T$
if $\sigma=-$.
\end{proposition}

\begin{proof}
We consider the case $\sigma=+$, that is, $\theta(u)=\theta(-u)$.
Since $T\circ \mb c(-u)=\mb c(u)$
and $T\in \op{SO}(3)$ (cf. Remark \ref{rem:key}), 
$$
-T\mb e(-u)=\mb e(u),\quad
T\mb n(-u)=\mb n(u), \quad
\mb b(u)=-T \mb b(-u).
$$
In the case of (1) (resp. (2)),
we set $\phi(u,v):=(-u,-v)$ (resp. $\phi(u,v):=(-u,v)$).
Then $A\circ \phi(u,v)=A(u,v)$ and $B\circ \phi(u,v)=-B(u,v)$
hold, and so
\begin{align*}
T\circ f_\theta\circ \phi
&=\mb c+
(A,-B)
\pmt{\cos \theta & -\sin \theta \\
       \sin \theta & \cos \theta }
\pmt{\mb n \\ -\mb b} \\
&=
\mb c+
(A,B)
\pmt{\cos \theta & \sin \theta \\
       -\sin \theta & \cos \theta }
\pmt{\mb n \\ \mb b} =f_{-\theta},
\end{align*}
proving the relation 
$f_\theta=T\circ f_{-\sigma \theta}\circ \phi$.
The case $\theta(u)=-\theta(-u)$ is proved in
the same way.

We then consider the case that $\sigma=1$.
In this case, 
$f_\theta=T\circ f_{-\theta}\circ \phi$
holds．Since $T$ is an isometry of $\R^3$,
we have $\phi^*ds^2_f=ds^2_g$, where
$f:=f_{-\theta}$ and $g=f_{-\theta}$.
So $g$ is isometric to $f$.
Since the cuspidal angle of $g$ takes
the opposite sign of that of $f$,
the image of $g$ does not coincide with $f$.
So $g$ is a faithful isomer of $f$.
Then the uniqueness of the faithful isomer of $f$
(cf. Theorem \ref{thm:S1})
yields that $g$ is right equivalent to $\check f$.
\end{proof}

Similarly, the following assertion holds.

\begin{proposition}\label{cor:V0}
Let $C$ be a space curve which 
admits a non-trivial symmetry 
$T\in \op{O}(3)\setminus \op{SO}(3)$
at $\mb 0$,
and let $f:=f_\theta\in \mc G^\infty(\R^2_J,\R^3,C)$
be a generalized cuspidal edge
as in the formula \eqref{eq:repF} 
such that
\begin{itemize}
\item $T\circ f(-u,0)=f(u,0)$, and
\item the cuspidal angle $\theta$
satisfies $\theta(u)=\sigma \theta(-u)$, where $\sigma\in \{+,-\}$.
\end{itemize}
Suppose that 
$A(u,v)$ and $B(u,v)$ satisfy one of the following two conditions:
\begin{enumerate}
\item $A(-u,-v)=A(u,v)$ and
$B(-u,-v)=B(u,v)$,
\item $A(-u,v)=A(u,v)$ and $B(-u,v)=B(u,v)$.
\end{enumerate}
Then $f_\theta=T\circ f_{\sigma \theta}\circ \phi$
holds, where $\phi(u,v)=(-u,-v)$ $($resp. $\phi(u,v)=(-u,v))$
in the case of {\rm (1)} $($resp. {\rm (2)}$)$. 
In particular, $f_{-\theta}$ is right equivalent to
$\check f$ 
if $\sigma=-$, and
the image of $f$
is invariant under $T$
if $\sigma=+$.
\end{proposition}

\begin{proof}
Like as in the case of the proof of
Proposition \ref{thm:V0},
$-T\mb e(-u)=\mb e(u)$ and $T\mb n(-u)=\mb n(u)$ hold.
Since $\det(T)=-1$, we have $T\mb b(-u)=\mb b(u)$.
In the case of (1) (resp. (2)), we set $\phi(u,v):=(-u,-v)$ (resp. $\phi(u,v):=(-u,v)$),
then the relation
$f_\theta=T\circ f_{\sigma \theta}\circ \phi$ is obtained 
like as in the case of the proof of Proposition~\ref{thm:V0}.
One can also obtain the last assertion 
imitating the corresponding argument in
the proof of 
Proposition \ref{thm:V0}.
\end{proof}

\begin{example}\label{ex:std}
Let $a,b$ be real numbers so that $a>0$ and $b\ne 0$.
Then 
$$
\mb c(u)
:=\left(
a \cos \left(\frac{u}{c}\right)-a,\,\,
a\sin \left(\frac{u}{c}\right),\,\,
\frac{bu}{c}\right)
\qquad (u\in \R)
$$
gives a helix of constant curvature $\kappa:=a/c^2$ and 
constant torsion $\tau:=b/c^2$, where $c:=\sqrt{a^2+b^2}$.
At the point $\mb 0:=\mb c(0)$ on the helix,
$\mb c$ satisfies 
$
T(\mb c(\R))=\mb c(\R),
$
where $T\in \op{SO}(3)$ is the $180^\circ$-rotation
with respect to the line passing through the origin $\mb 0$
which is parallel to the principal normal vector
$\mb n(0)$. We set $a=b=1$,  $\theta=\pi/4$.
By setting
$$
(A_1,B_1):=(v^2,v^3),\quad
(A_2,B_2):=(v^2,v^5),\quad
(A_3,B_3):=(v^2,uv^3).
$$
The surfaces
$g_{i,\pm}:=f_{\pm \pi/4}$ 
($i=1,2,3$)
associated to the Fukui data
($\mb c,\pm \pi/4, A_i,B_i$)
correspond to
cuspidal edges,
$5/2$-cuspidal edges, and cuspidal cross caps,
respectively.
The first two cases satisfy (1) of
Proposition \ref{thm:V0}
and the third case satisfies (2) 
of Proposition \ref{thm:V0}.
So $g_{i,-}$ ($i=1,2,3$)
is a faithful isomer of $g_{i,+}$.
\end{example}

\begin{figure}[th]
\begin{center}
        \includegraphics[height=2.3cm]{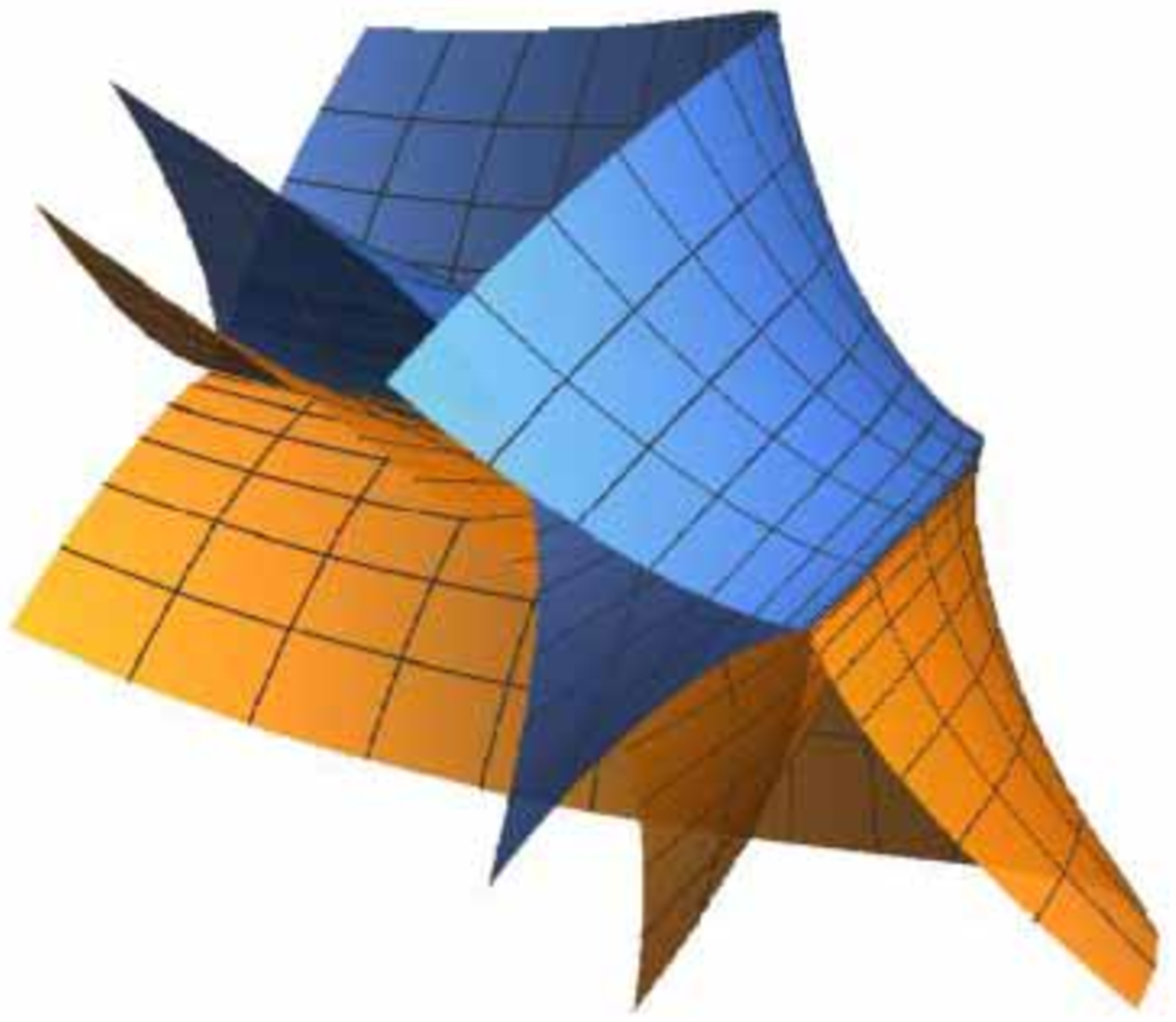} \quad
        \includegraphics[height=2.3cm]{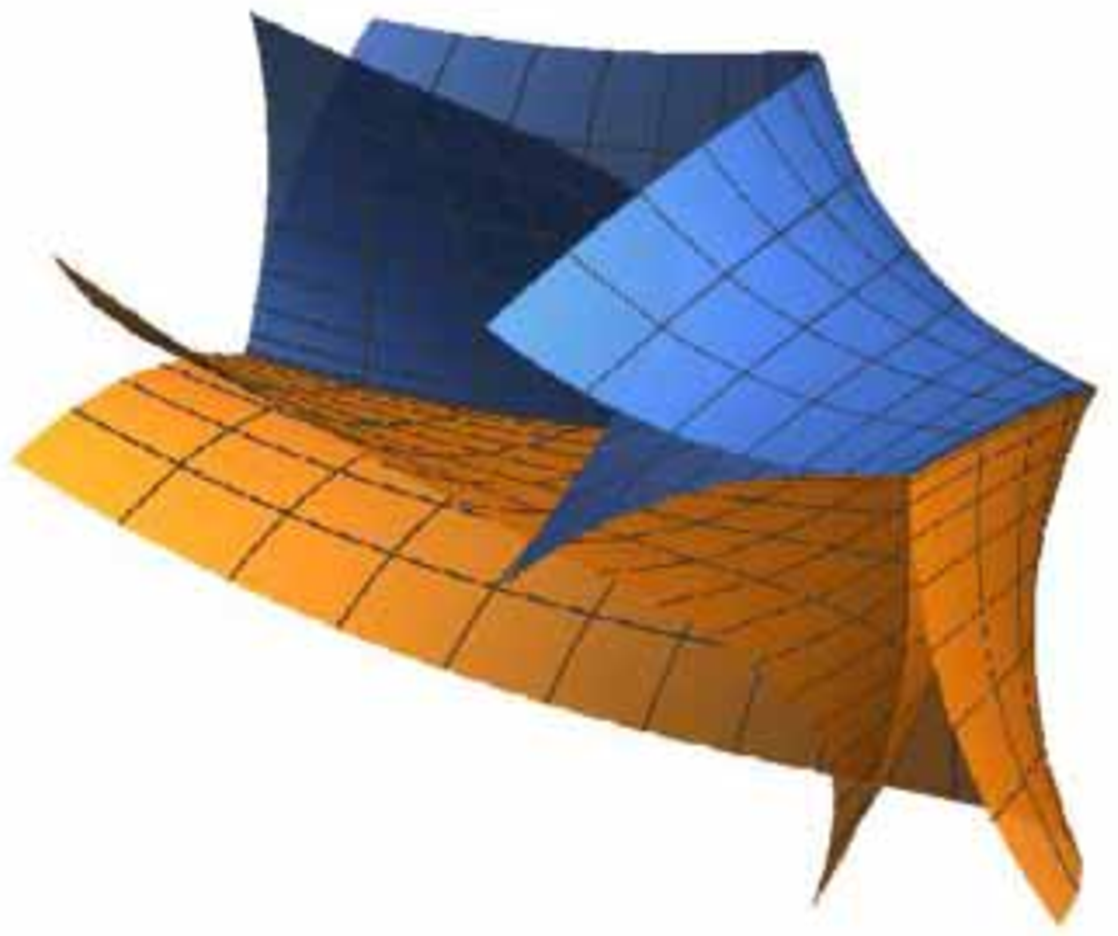}\quad \,\,
        \includegraphics[height=2.7cm]{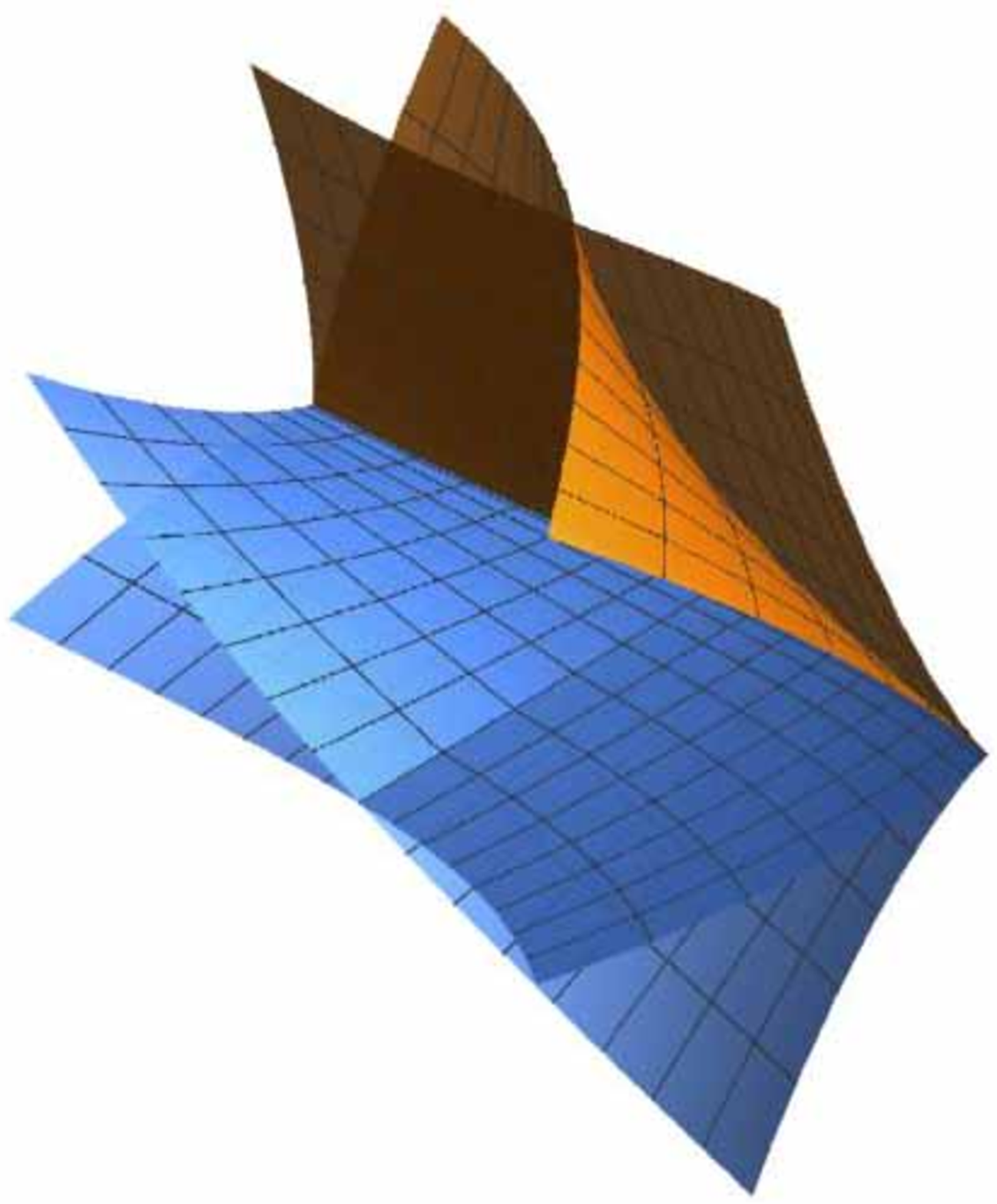}
\caption{The images of cuspidal edges $g_{1,\pm}$ (left),
$5/2$-cuspidal edges $g_{2,\pm}$ (center) 
and cuspidal cross caps $g_{3,\pm}$ (right) given in Example \ref{ex:std}.
(The orange surfaces correspond to $g_{i,+}$ and
the blue surfaces correspond to $g_{i,-}$ for $i=1,2,3$.)
}
\label{fig:steady} 
\end{center}
\end{figure}

Finally, we consider the case of fold singularities:

\begin{example}
We let $\mb c(u)$ be a $C^\infty$-regular space curve 
with positive curvature $\kappa$ and torsion $\tau$.
If we set
$$
g_\pm (u,v):=\mb c(u)+\frac{v^2}{2}(\cos\theta \mb n(u)\mp 
\sin\theta \mb b(u)),
$$
then it can be easily checked that
$g_-$ is a faithful isomer of $g_+$, where $\theta$ is a constant. 
These two surfaces can be
extended to the following regular ruled surfaces:
$$
\tilde g_\pm=\mb c(u)+\frac{v}{2}(\cos\theta \mb n(u)\mp \sin\theta \mb b(u)). 
$$
\end{example}

\appendix
\section{A representation formula for generalized cusps}\label{app1}

A plane curve $\sigma:J\to \R^2$
is said to have a singular point at $t=t_0$ if $\dot\sigma(t_0)=\mb 0$
(the dot means $d/dt$).
The singular point $t=t_0$ is called a {\it generalized cusp}
if 
$
\ddot \sigma(t_0)\ne \mb 0.
$
In this situation, it is well-known that
\begin{enumerate}
\item[(i)]
$t=t_0$ is a cusp if and only if 
$\ddot\sigma(t_0),\dddot\sigma(t_0)$ are linearly independent,
\item[(ii)](cf. \cite{P})
$t=t_0$ is a $5/2$-cusp if and only if 
$\ddot\sigma(t_0),\dddot\sigma(t_0)$ are linearly dependent
and
$$
3\op{det}(\ddot\sigma(t_0),\sigma^{(5)}(t_0))\ddot\sigma(t_0)-10
\op{det}(\ddot\sigma(t_0),\sigma^{(4)}(t_0))
\dddot\sigma(t_0)\ne \mb 0.
$$
\end{enumerate}
From now on, we set $t_0=0$.
The arc-length parameter $s(t)$ of $\sigma$ given by
$$
s(t):=\int_0^t |\dot \sigma(u)|du
$$
is not smooth at $t=0$, but if we set
$
w:=\op{sgn}(t)\sqrt{|s(t)|},
$
then this gives a parametrization of $\sigma$
near $t=0$, which is called the {\it half-arc-length parameter}
of $\sigma$ near $t=0$ in \cite{SU}.
However, for our purpose,
as Fukui \cite{F} did,
the parameter  
\begin{equation}
v:=\sqrt{2}w=\op{sgn}(t)
\left(2\int_0^t |\dot \sigma(u)|du\right)^{1/2}
\end{equation}
called the {\it normalized half-arc-length parameter}
is convenient, since it is compatible with the property
$|f_{vv}|=1$ for adapted coordinate systems (cf. Definition 
\ref{def:adapted})  of generalized cuspidal edges. 
This normalized half-arc-length parameter can be
characterized by the property that
$v^2/2$ gives the arc-length parameter of $\sigma$.
Then by \cite[Theorem 1.1]{SU}, we can write
\begin{equation}\label{A2}
\sigma(v)=\int_0^v u (\cos \theta(u),\sin \theta(u))du,
\qquad \theta(v)=\int_0^v \hat \mu(u)du.
\end{equation}
We need the following lemma, which can be proved by
a straightforward computation.

\begin{lemma}\label{lem:A1}
Let $v$ be the normalized half-arc-length parameter
of the generalized cusp $\sigma(w)$ at $w=0$.
Then there exists an orientation preserving isometry $T$ of $\R^2$
such that
\begin{equation}\label{eq:Tsigma}
T\circ \sigma(v)=
\Big (
\frac{v^2}2 - \frac{\mu_0^2 v^4}8 - \frac{\mu_0 \mu_1v^5}{10}, 
\frac{\mu_0 v^3}3 + \frac{\mu_1 v^4}8 + 
  \frac{(-\mu_0^3 + 2\mu_2)v^5}{30} 
\Big)+o(v^5),
\end{equation}
where
$$
\hat\mu(v)=\sum_{j=0}^{2} \mu_jv^j+o(v^3),
$$
and $o(v^5)$ $($resp. $o(v^3))$ is a term higher than $v^5$
$($resp. $v^3)$. 
\end{lemma}

Using this with (i) and (ii), one can easily obtain the following assertion:

\begin{proposition}
Let $v$ be the normalized half-arc-length parameter
of the generalized cusp $\sigma(w)$ at $w=0$.
Then 
\begin{enumerate}
\item $w=0$ is a cusp of $\sigma$ if and only if
$\mu_0 \ne 0$, and 
\item $w=0$ is a $5/2$-cusp of $\sigma$ if and only if
$\mu_0 =0$ and $\mu_2\ne 0$.
\end{enumerate}
\end{proposition}

It is remarkable that the coefficient $\mu_1$ does not affect
the criterion for $5/2$-cusps.
In this case, $\mu_0=0$ holds, 
and $\mu_1$ and $\mu_2$ are proportional to
the \lq\lq secondary cuspidal curvature\rq\rq\ and
the \lq\lq bias\rq\rq\ of $\sigma(t)$ at $t=0$, respectively.
Geometric meanings for these two invariants for 
$5/2$-cusps can be found in \cite[Proposition 2.2]{HS}.

\begin{ack}
The authors thank Toshizumi Fukui and Wayne Rossman
for valuable comments.
\end{ack}

\end{document}